\documentclass[twoside]{amsart}
\usepackage{a4wide, amsmath, amsthm, amssymb, amscd, mathptmx, mathtools}
\usepackage{hyperref}
\usepackage{enumitem}
\usepackage{braket} % For command \Set{|}
\usepackage{soul,xcolor} %For colored strikeout
\setstcolor{red} %strikeout color
\usepackage[colorinlistoftodos, textwidth=2.7cm]{todonotes}
\setuptodonotes{fancyline, color=orange}
\usepackage{setspace} %so that "spacing" environment from tinytodo works
\usepackage{amsfonts}
\usepackage{bbm}
\usepackage{tikz}
\usetikzlibrary{arrows}

\newcommand\imCMsym[4][\mathord]{%
	\DeclareFontFamily{U} {#2}{}
	\DeclareFontShape{U}{#2}{m}{n}{
		<-6> #25
		<6-7> #26
		<7-8> #27
		<8-9> #28
		<9-10> #29
		<10-12> #210
		<12-> #212}{}
	\DeclareSymbolFont{CM#2} {U} {#2}{m}{n}
	\DeclareMathSymbol{#4}{#1}{CM#2}{#3}
}

\imCMsym{cmmi}{124}{\Cjmath}

\usepackage{color}
\usepackage{pstricks}

\usepackage[all]{xy}\CompileMatrices\SelectTips{cm}{12}

 \makeatletter 
 \providecommand\@dotsep{5} 
 \makeatother 
\theoremstyle{plain}
\newtheorem{Thm}{\sc Theorem}[section]
\newtheorem{Theorem}[Thm]{\sc Theorem}
\newtheorem{Corollary}[Thm]{\sc Corollary}

\newtheorem*{Corollary*}{\sc Corollary}

\newtheorem{Proposition}[Thm]{\sc Proposition}
\newtheorem*{Proposition*}{\sc Proposition}
\newtheorem{Lemma}[Thm]{\sc Lemma}

\theoremstyle{definition}
\newtheorem{Definition}[Thm]{Definition}

\theoremstyle{remark}
\newtheorem{Remark}[Thm]{Remark}
\newtheorem{Example}[Thm]{Example}
\newtheorem*{Example*}{Example}
\newtheorem*{Remark*}{Remark}

\renewcommand{\AA}{{\mathbb A}}
\newcommand{\CC}{{\mathbb C}}

\newcommand{\DD}{{\mathbb D}}
\newcommand{\EE}{{\mathbb E}}
\newcommand{\FF}{{\mathbb F}}
\newcommand{\GG}{{\mathbb G}}

\newcommand{\LL}{{\mathbb L}}

\newcommand{\NN}{{\mathbb N}}
\newcommand{\ZZ}{{\mathbb Z}}

\renewcommand{\SS}{{\mathbb S}}

\newcommand{\cB}{{\mathcal B}}
\newcommand{\cC}{{\mathcal C}}
\newcommand{\cD}{{\mathcal D}}
\newcommand{\cX}{{\mathcal X}}

\newcommand{\cF}{{\mathcal F}}

\newcommand{\cO}{{\mathcal O}}

\newcommand{\cU}{{\mathcal U}}

\newcommand{\sC}{{\mathfrak C}}

\newcommand{\sF}{{\mathfrak F}}

\newcommand{\sT}{{\mathfrak T}}

\newcommand{\Aff}{{\rm Aff}}
\newcommand{\Aut}{{\rm Aut}}

\newcommand{\Pic}{{\mathop{\rm Pic \, }}}
\newcommand{\GL}{\mathop{\rm GL\, }}

\newcommand{\Coh}{{\mathop{\operatorname{Coh}\, }}}

\newcommand{\Mic}{{\mathop{\operatorname{MIC}\, }}}

\newcommand{\Hom}{{\mathop{{\rm Hom}}}}
\newcommand{\et}{{\mathop{\rm \acute{e}t \, }}}

\newcommand{\cHom}{{\mathop{{\mathcal H}om}}}

\newcommand{\id}{\mathop{\rm Id}}

\newcommand{\im}{\mathop{\rm im \, }}

\newcommand{\cEnd}{{\mathop{{\mathcal E}nd}\,}}

\newcommand{\coker}{{\mathop{\rm coker \, }}}

\newcommand{\rk}{{\mathop{\rm rk \,}}}

\newcommand{\Spec}{{\mathop{{\rm Spec\, }}}}

\newcommand{\Vect}{{\mathop{{\rm Vect \,}}}}

\def \Coh {{\sf Coh}}
\def \Gal {{\sf Gal}}
\def \Reff {{\sf Ref}}
\def \perf {{\sf perf}}
\def \Vect {{\sf Vect}}

\def \Ker {{\sf Ker}}
\def \pp {{\mathfrak{p}}}
\def \Rep {{\sf Rep}}

\def \op {{\rm op}}
\def \PiF {\Pi^{F\sf{\text{-}div}}}
\def \Hl {{\rm H}}
\def \DD {{\rm DD}}
\def \No {{\rm N}}
\def \et {{\textup{ét}}}

\newcommand\Fdiv {{F\sf{\text{-} div}}}
\newcommand{\unit}{{\mathbbm{ 1}}}

\newcommand{\End}{{\mathop{{\rm End \,}}}}

\begin{document}
	
	\markboth {$F$-divided bundles}{$F$-divided bundles}

	\title{$F$-divided bundles  on normal $F$-finite schemes}

\author{Adrian Langer}
\email{alan@mimuw.edu.pl}
\address{Institute of Mathematics, University of Warsaw,
	ul.\ Banacha 2, 02-097 Warszawa, Poland}
	
\author{Lei Zhang}
\email{cumt559@gmail.com}
\address{Sun Yat-Sen University\\
    School of Mathematics (Zhuhai)\\    
    Zhuhai, 
    Guangdong Province\\ China}
	
	\date{
    %August 21, 2025
        \today
    }
	
	%%%
	\maketitle
	%%%
	
	%%%%%%%%%%%%%%%%%%%%%%%%%%%%%%%%%%

	\begin{abstract}{ 
		In this paper we study $F$-divided bundles on  irreducible Noetherian normal $F$-finite $\FF_p$-schemes and we show that their Tannakian category is governed by the behaviour at the generic point. In particular, if $U\subset X$ is an open subset of a normal variety defined over an algebraically closed field then the corresponding homomorphism of $F$-divided fundamental groups is faithfully flat. This is analogous to a known fact about the topological fundamental group of an open subset of a normal complex analytic variety.
		We use this result to show that simply connected, proper, normal varieties in positive characteristic admit no nontrivial $F$-divided bundles. This generalizes an earlier result of H. Esnault and V. Mehta concerning smooth projective varieties, and settles Gieseker's conjecture in a more general setting.
		}
	\end{abstract}
	
%	\listoftodos[Todos]\relax
	
	\section*{Introduction}
	
Let $X$ be a unibranch complex analytic variety and let $U\subset X$ be the complement of a proper closed analytic subset. It is well known that in this case we have a surjective map $\pi_1^{\sf top}(U) \to \pi_1^{\sf top}(X)$
	of topological fundamental groups of $U $ and $X$ (see, e.g., \cite[(0.7) (B)]{Fulton-Lazarsfeld1981}). 
	This fact has the following  algebraic analogue.  A. Grothendieck introduced in \cite{Grothendieck1968} so called coherent stratified sheaves. On a scheme $X$ of finite type over a field $k$ these sheaves were later shown by N. Saavedra Rivano in \cite[Chapitre VI, 1.2]{Saavedra72} to form a Tannakian category. Upon a choice of a rational point this leads to a stratified fundamental group $\pi _1^{\sf strat}(X)$. 
	\cite[Expos\'e XIV, Corollaire 1.19]{SGA2} asserts that if $X$ is also
    geometrically irreducible and geometrically unibranch and $i:
    U\hookrightarrow X$ is an open subset then for all prime integers $l$ (or,
    equivalently, for  one prime $l$) the canonical map
    ${\underline{\ZZ/l\ZZ}}_U \to i_*i^* ({\underline{\ZZ/l\ZZ}}_U)$ is an
    isomorphism of \'etale sheaves. If $k=\CC$ (or in fact for any
    algebraically closed field of characteristic zero)  this fact, together
    with the proof of \cite[Theoreme 4.4]{Grothendieck1970} and  the
    Grothendieck--Malcev theorem, shows that  $\pi_1^{\sf strat}(U) \to
    \pi_1^{\sf strat}(X)$ is faithfully flat. 

    The main goal of this paper is to show that the same theorem holds in positive characteristic.
   Recall that an $\FF_p$-scheme is called $F$-finite if its absolute Frobenius endomorphism is a finite map. If $X$ is an $F$-finite Noetherian scheme, the category of coherent stratified sheaves (or crystals on the infinitesimal site $(X/\FF_p)_{\sf inf}$) is equivalent to that of $F$-divided bundles (see \cite[Theorem 2.1]{Bhatt} or \cite[Proposition 3.4]{Esnault-Srinivas2016} for a more special case).
   To any such scheme $X$, we associate a pro-smooth
banded affine gerbe $\Pi_X^{\Fdiv}$ over the maximal perfect subfield $\cO_X(X)^\perf$ of $\cO_X(X)$. This gerbe is characterized by the property that its category of vector bundles $\Vect(\Pi_X^{\Fdiv})$ is equivalent to the category of $F$-divided bundles $\Vect^\perf(X)$.
Our central result is the following theorem (see Theorems \ref{normal open
induces surjection on stratified fundamental groups}).

	\begin{Theorem}\label{main1}
		Let $X$ be an integral Noetherian geometrically unibranch $F$-finite
        scheme { over $\FF_p$}.
		Then  for any open subset $U\subseteq X$, the restriction
        $\Pi_U^{\Fdiv}\to\Pi_X^{\Fdiv}$ is a relative gerbe { over
        $\cO_U(U)^\perf=\cO_X(X)^\perf$}.
	\end{Theorem}
	
	In case $X$ is a unibranch variety over an algebraically closed field $k$
    of positive characteristic,  we have $\cO_X(X)^\perf=k$ and the above theorem says that the homomorphism  $\pi_1^{\sf strat}(U) \to \pi_1^{\sf strat}(X)$ of affine $k$-group schemes is faithfully flat. It is known that for the generic point $\eta $ of a normal, connected, Noetherian scheme $X$, the canonical homomorphism $\pi_1^{\et}(\eta) \to \pi_1^{\et}(X)$
	of \'etale fundamental groups is surjective  (see \cite[Expos\'e V,
    Proposition 8.2]{SGA1}). In our case, we show that one can similarly
    replace an open subset $U\subset X$ with the generic point of $X$. { We
        also establish a generalization of this result to algebraic stacks (see
        Theorem \ref{normal open
induces surjection on stratified fundamental groups stack case})). However, for simplicity, we restrict our discussion in this introduction to the case of schemes.
}
	
	When the scheme $X$ is regular, the category of $F$-divided bundles is equivalent to the category of $D$-modules (see Section 4), and the above theorem can be reformulated as a statement about $D$-modules. In the special case where $X$ is a smooth variety over an algebraically closed field, this result was proven by L. Kindler in \cite{Kindler2015} using $D$-module techniques.
	
	The above theorem has very concrete new applications, even in the case of
    smooth varieties. Our main application is to prove the following theorem
    ({ see Theorem \ref{EM10 for normal proper}}), which provides a
    positive answer to Gieseker's conjecture \cite[p.~8]{Gieseker1975} for
    { normal} proper varieties.

	\begin{Theorem}\label{main2}	
		Let $X$ be a normal scheme that is proper and geometrically connected over a perfect field $k$ of positive characteristic.
		If for some rational point  $x\in X(k)$ the maximal \'etale quotient  	\(\pi_1^{\No,\et}(X,x)\) of the Nori fundamental group scheme vanishes 	then there are no non-trivial $F$-divided bundles on $X$.	
	\end{Theorem}

	In case $X$ is smooth and projective the above theorem was proven by H. Esnault and V. Mehta  in \cite{Esnault-Mehta2010}. However, the proof of this theorem used moduli spaces of semistable sheaves and therefore it was restricted to the projective case.
	
	\medskip
	
	The proofs of the above theorems use various deep results. Even though Theorem \ref{main1} (or rather its corollary for varieties) could have been formulated as early as the 1970s, its proof was out of reach at the time.  To prove this theorem, we rely in particular on Gabber's generalization of  de Jong's alteration theorem,  which ensures the existence of smooth alterations in a broader setting.
	We also make use of a recent theorem of B. Bhatt on $h$-descent of $F$-divided bundles, proved using derived methods.
	
	To prove Theorem \ref{main2} we invoke Theorem \ref{main1} and establish the structure of simple $F$-divided bundles on normal projective varieties (see Proposition \ref{Jordan-Holder-for-F-divided-bundles}). Next, we apply Chow's lemma to reduce the problem to the study of $F$-divided bundles on a normal projective variety, which we examine in detail in this paper. The proof of Theorem \ref{main2} subsequently follows a strategy analogous to that of \cite[Theorem 1.1]{Esnault-Mehta2010}. However, an additional challenge emerges when one needs to descend certain vector bundles from a projective variety $Y$  to a proper variety $X$. In general, the behaviour of sheaves on proper varieties can be quite intricate, as the Hilbert functor does not need be representable  by a scheme, only an algebraic space, and its connected components need not be of finite type. As a result, it is unclear whether vector bundles on $Y$ that descend to $X$  form a locally closed subset in the moduli space. Fortunately, we can address this problem using some nonflat descent, which is ``the most delicate part'' (see \cite[Expos\'e XII, Section 4]{SGA6}) of the proof of relative representability of the Picard's functor.

\medskip
	
The structure of the paper is as follows. In the first section we review the theory of gerbes and Tannakian categories. In Section 2 we show a few general results on $F$-divided sheaves, proving in particular Bhatt's theorem on $h$-descent of $F$-divided sheaves. In the next section we study local properties of $F$-divided bundles on normal schemes and prove a key technical result showing that their Picard group injects into the Picard group at the generic point. In Section 4 we relate $F$-divided bundles on regular schemes to $D$-modules. This provides a different path to some of our results in the regular case (see Remark \ref{Kindler-works-nonetheless}). In the following section we prove Theorem \ref{main1}. Section 7 contains a description of simple $F$-divided bundles on normal projective varieties, generalizing \cite[Proposition 2.3]{Esnault-Mehta2010} and providing a simpler proof.
In the last section we use  our previous results to prove Gieseker's conjecture for normal proper varieties.

\subsection{Notation}

Let $X$ be a { locally} Noetherian scheme. A vector bundle on $X$ is a coherent
$\cO_X$-module, which is locally free (note that with this definition the rank
of a vector bundle can vary on different connected components).  
For an integral scheme $X$ we write $K(X)$ for the function field of $X$.

{ We denote $\Coh_{\ZZ}$ (resp. $\Vect_{\ZZ}$) the stack of finitely
    presented sheaves (resp. vector bundles) over
    the fpqc site
    $\Aff/\ZZ$ of affine schemes. Let $\cX$ be a category fibered over
    $\Aff/k$. We will interchangeably use the following equivalent definitions of
    finitely presented sheaves (resp. vector bundles) on $\cX$:
    \begin{itemize}
        \item
    A 1-morphism $\cX\to
    \Coh_{\ZZ}$ (resp. $\cX\to \Vect_{\ZZ}$) over $\Aff/\ZZ$; 
    \item
       A functorial association $\xi\mapsto M_{\xi}$: for any affine scheme
    $\Spec(R)$ and any map $\xi\colon\Spec(R)\to\cX$, one associates a finitely
    presented (resp. finitely presented projective)
    $R$-module $M_{\xi}$, and the association is functorial in the
    obvious sense;
    \item Suppose $\cX$ is equipped with a representable fpqc-covering from a scheme
    $X\twoheadrightarrow \cX$. Then a finitely presented module (resp. a vector
    bundle) on $\cX$ is a finitely presented module (resp. a vector bundle) $\sF$ on $X$
    together with an isomorphism $\phi\colon
    p_1^*\sF\xrightarrow{\cong}p_2^*\sF$ on $X\times_{\cX}X$ satisfying the
    cocycle condition $p_{23}^*(\phi) \circ p_{12}^*(\phi) = p_{13}^*(\phi)$.
    \end{itemize} 
    The category of finitely presented modules (resp. vector bundles) on $\cX$ is
    denoted by $\Coh(\cX)$ (resp. $\Vect(\cX)$). If $\cX$ is a locally Noetherian algebraic stack, then
    $\Coh(\cX)$ is just the category of coherent sheaves on $\cX$. If
    $\cX=\cB_kG$ is an affine gerbe, then we usually choose the rational point
    $\Spec(k)\to\cB_kG$, which corresponds to the trivial $G$-torsor, as the
    representable fpqc-covering. In this case, a vector bundle or a finitely
    presented sheaf on $\cB_kG$ is
    nothing but a finite dimensional $k$-vector space $V$ together with a
    $G$-action, i.e., a finite dimensional $G$-representation. 
}

\section{Generalities on gerbes and Tannakian categories}

This section reviews the theory of gerbes, which provide the natural geometric
framework for our results. The key idea is that certain
algebraic varieties with no rational points, while lacking a Tannakian fundamental
group, can still be endowed with a fundamental gerbe. We will be particularly
interested in affine gerbes, which are intimately connected to Tannakian
categories. We refer the reader to  \cite{Deligne-Milne1982}, \cite{BV15} and \cite{TZ1} for  additional details and explanation.

\subsection{Gerbes}
Affine gerbes form an important class of stacks. Serving as a natural
generalization of affine group schemes, they provide the correct framework for
studying moduli problems with non-unique isomorphisms. Their geometry is
controlled by their representation categories, making them indispensable tools
in the study of Tannakian fundamental group schemes.

\begin{Definition}\label{def:gerbe}
	Let $\sC$ be a site. A stack $\cX$ fibered in groupoids over $\sC$ is a \emph{gerbe} if:
	\begin{enumerate}[label=(\roman*)]
		\item (\textbf{Local existence}) For every $U \in \sC$, there exists a covering $\{U_i \to U\}$ and objects $x_i \in \cX(U_i)$.
		\item (\textbf{Local connectivity}) For every $U \in \sC$ and any two objects $x, x' \in \cX(U)$, there exists a covering $\{U_i \to U\}$ such that $x|_{U_i} \cong x'|_{U_i}$ in $\cX(U_i)$.
	\end{enumerate}
	A gerbe is thus a stack that is locally non-empty and locally connected by
    isomorphisms. This is reminiscent of a torsor -- indeed, gerbes can be thought of as ``2-torsors'' for a sheaf of groups.
\end{Definition}

For any object $T \in \cX(U)$, the sheaf $\Aut(T)$ of automorphisms of $T$ is a central object of study. The local connectivity condition implies that for any two objects $T, T' \in \cX(U)$, the sheaves $\Aut(T)$ and $\Aut(T')$ are locally isomorphic. This leads to the definition of a band, which captures the isomorphism class of these automorphism sheaves. For our purposes, the following geometric perspective is most useful.

\begin{Definition}\label{def:affine-gerbe}
	Let $k$ be a field and $\sC = \Aff/k$ the fpqc site of affine $k$-schemes.
    A gerbe $\cX$ over $\sC$ is called \emph{affine} (resp.
    {\emph{pro-smooth banded}})
     if there exists a field extension $K/k$ and an object $T \in \cX(K)$ such
     that the automorphism group scheme $\Aut(T)$ is representable by an affine
     (resp. pro-smooth) $K$-group scheme.
\end{Definition}

By fpqc descent, if $\cX$ is an affine gerbe, then for any $U \in \Aff/k$ and any $T \in \cX(U)$, the group functor $\Aut(T)$ is representable by an affine, flat $U$-group scheme. 

The prototypical example of an affine gerbe is a classifying stack.
Let $G$ be an affine group scheme over a field $k$. The classifying stack $\cB_kG$, which associates to a $k$-algebra $R$ the groupoid of fpqc $G$-torsors over $\Spec(R)$, is a gerbe over $\Aff/k$. This gerbe admits a $k$-rational point corresponding to the trivial $G$-torsor over $\Spec(k)$.
Conversely,  any affine gerbe with a rational point is a classifying stack.
More precisely, if $\Gamma$ is an affine gerbe over $\Aff/k$ which admits a
rational point $x \in \Gamma(k)$, then $\Gamma \simeq \cB_kG$, where $G =
\Aut(x)$ is the automorphism group scheme of $x$, and this equivalence sends
$x$ to the trivial $G$-torsor. Any gerbe that is equivalent to $\cB_kG$ for
some affine $k$-group scheme $G$ is called a \emph{trivial gerbe}.
Since every affine gerbe $\Gamma$ over $\Aff/k$ admits a section fpqc-locally,
it is fpqc-locally a trivial gerbe. That is, for some fpqc field extension
$l/k$, we have $\Gamma_l := \Gamma \times_k l \simeq \cB_l G$ for some affine
group scheme $G$ over $l$. Note that a different trivialization may yield a
different group scheme and they differ by an inner twist; the intrinsic object is the gerbe itself, not a specific group presenting it.

\begin{Definition}\label{def:relative-gerbe}
	Let $\phi\colon \Gamma_1 \to \Gamma_2$ be a 1-morphism of affine gerbes over the fpqc site $\Aff/k$. The morphism $\phi$ is called \emph{a relative gerbe} if for some (hence for all) field extension $l/k$ and a 1-morphism $\Spec(l) \to \Gamma_2$, the fibered product $\Gamma_1 \times_{\Gamma_2} l \to \Spec(l)$ is a gerbe over the fpqc site $\Aff/l$.
\end{Definition}

Let $\varphi\colon G \to G'$ be a homomorphism of affine $k$-group schemes.
It induces a 1-morphism of classifying stacks $\phi\colon \cB_kG \to \cB_kG'$.
\begin{enumerate}[label=(\alph*)]
	\item If $\varphi$ is faithfully flat, then for the point $x\colon \Spec(k) \to \cB_kG'$ corresponding to the trivial $G'$-torsor, the fiber $\cB_kG \times_{\cB_kG'} k$ is equivalent to $\cB_k(\Ker(\varphi))$, which is a gerbe over $\Aff/k$.
	\item If $\varphi$ is a closed immersion, then the fiber $\cB_kG \times_{\cB_kG'} k$ is represented by the quotient scheme $G'/G$, which is a nontrivial $k$-scheme unless $G' = G$.
\end{enumerate}
Since any homomorphism of affine group schemes factors as a faithfully flat
quotient map followed by a closed immersion, the induced morphism $\phi$ is a
relative gerbe (in the sense of Definition~\ref{def:relative-gerbe}) if and
only if $\varphi$ is faithfully flat. This shows that the concept of a relative
gerbe generalizes the notion of a surjective homomorphism to the context of
morphisms between affine gerbes.

\subsection{Tannakian gerbes}\label{Tannakian Duality} A Tannakian category is a rigid abelian
     tensor category which is morally the category of representations of an affine group scheme. The presence or absence of a fiber functor over the base field determines whether this group scheme exists or if one must work with the more general notion of an affine gerbe. The fundamental link between gerbes and
    Tannakian categories is provided by the following principle:

    \begin{Theorem} Let $k$ be a field. There is an equivalence of $2$-categories:
\[ 
    \fbox{
  \begin{tabular}{ccc}
            &\text{Affine gerbes}&\\
      &\text{over $k$}&
  \end{tabular}
}
\Longleftrightarrow
\fbox{
  \begin{tabular}{ccc}
            &\text{Tannakian categories}&\\
            &\text{over $k$}&
  \end{tabular}
}
\]
This equivalence is implemented by the following constructions:
\begin{itemize}
    \item Given an affine gerbe $\Gamma$ over $k$, its \emph{category of representations} $\Rep(\Gamma)$ (i.e., the category $\Vect(\Gamma)$ of vector bundles on $\Gamma$) is a Tannakian category over $k$.
    \item Conversely, given a Tannakian category $\sT$ over $k$, the functor
    \[
        \Pi^{\sT} \colon (\Aff/k)^\op \to \mathrm{(Groupoids)}, \  T \mapsto
        \left\{ \omega \colon \sT \to \Vect(T) \text{ faithful, exact,
        $k$-linear, tensor functors} \right\}
    \]
    is an affine gerbe over $k$, called its \emph{fundamental gerbe}.
\end{itemize}
These constructions are quasi-inverse to each other.   
\end{Theorem}

\begin{proof}
    The proof is carried out in three steps:
    \begin{enumerate}[label=(\arabic*)]
        \item Well-definedness of the two functors:
            \begin{enumerate}
                \item If $\Gamma$ is an affine gerbe over $\Aff/k$, then
                    $\Vect(\Gamma)$ forms a $k$-Tannakian category;
                \item If $\sT$ is a $k$-Tannakian category, then the 2-functor
                    $\Pi^{\sT}$
                    is an affine gerbe over $\Aff/k$.
            \end{enumerate}
        \item \emph{Tannakian recognition}: Let $\sT$ be a $k$-Tannakian
            category. The natural
            functor $\sT\to\Vect(\Pi^{\sT})$ is an equivalence.
        \item \emph{Tannakian reconstruction}: Let $\Gamma$ be an affine gerbe and $\cX$ 
            any fibered category over $\Aff/k$. The natural pullback
            functor induces an equivalence: 
            \[
                \Hom_{\Aff/k}(\cX,\Gamma)\longrightarrow
                \Hom_{\otimes,k}(\Vect(\Gamma),\Vect(\cX))
            \]
            where $\Hom_{\otimes,k}$ denotes the category of faithful, exact,
            $k$-linear, tensor functors. 
    \end{enumerate}
    (1).(a) is due to the fact that $\mathrm{H}^0(\cO_{\Gamma})=k$ and
    $\Vect(\Gamma)=\Coh(\Gamma)$. Indeed, for any field extension $l/k$, we
    have
    $\mathrm{H}^0(\cO_{\Gamma\times_kl})=\mathrm{H}^0(\cO_{\Gamma})\otimes_kl$
    and $\cF\in\Coh(\Gamma)$ is a vector bundle iff
    $\cF\otimes_kl\in\Vect(\Gamma\times_kl)$. Since any gerbe becomes trivial
    after some field extension $l/k$,
    it suffices to prove the statements for the trivial gerbe.
    For $\Gamma=\cB_kG$, the trivial section $\Spec(k)\to\cB_kG$ is faithfully
    flat and $k\subseteq\mathrm{H}^0(\cO_{\cB_k(G)})$, implying
    $\mathrm{H}^0(\cO_{\cB_k(G)})=k$; For any $M\in\Coh(\cB_kG)$, the pullback
    of the $R$-module $M_{\xi}$ by the $G$-torsor
    $\xi\colon\Spec(R)\to\cB_kG$ is free, hence $M$ is a vector
    bundle. (1).(b) is contained
    in \cite[Theorem 3.2]{Deligne-Milne1982}. Tannakian recognition is
    \cite[Theorem 3.9]{Deligne-Milne1982}. For Tannakian reconstruction, see
    \cite[Theorem 1.4 and Example 1.5]{TZ1}.
\end{proof}

Recall that a \emph{neutral affine gerbe} over $\Aff/k$ is a pair $(\Gamma,*)$, where
$\Gamma$ is an affine gerbe and $*$ is a $k$-rational section of $\Gamma$. A
\emph{neutral Tannakian category} is a pair $(\sT,\omega)$, where $\sT$ is a
$k$-Tannakian category and $\omega\colon\sT\to\Vect(k)$ is a fiber functor. As
a corollary of Theorem \ref{Tannakian Duality}, we have
\begin{Corollary} There are natural equivalences of $2$-categories:
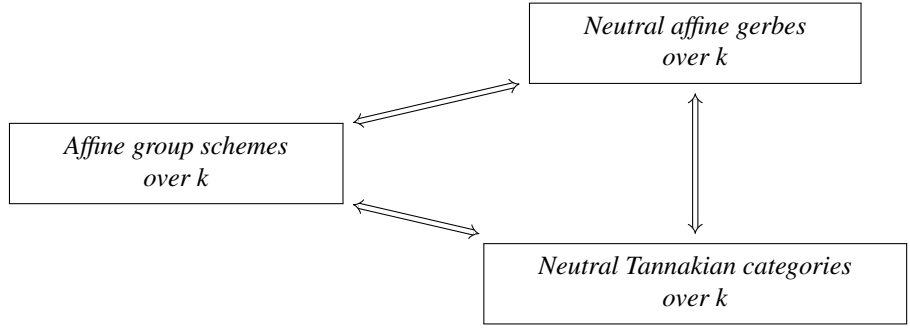
\[
    %%%%%% Arrow tip style [-stealth], [-to], [-latex], [-|], [>=latex,->], [|->]
    %%%%%% Arrow style: [red], [dashed], [dotted, line width=12pt], [ultra thick], [thick], [semithick], [thin], [very thin], [ultra thin]
    %%%%%% Double arrows: \draw ([yshift=-2pt] A0_0.east) -- ([yshift=-2pt] A0_1.west);
    %%%%%% Edge follow the rules of draw.
    %%%%%% Node style: [above,scale=0.9], [below], [midway], [anchor=south], [ellipse,draw], [circle,fill=red!20] 
    %%%%%% Arrow Node: [pos=0.5] (this is the default position)
    %%%%%% Equation tag position center alignment: [baseline=(current  bounding  box.center)]
       \begin{tikzpicture}[xscale=6.9,yscale=-3.2]
            \node (A0_0) at (0, 0.5) {$\fbox{
  \begin{tabular}{ccc}
            &\text{Affine group schemes}&\\
      &\text{over $k$}&
  \end{tabular}
}$};
            \node (A0_1) at (1, 0) {$\fbox{
  \begin{tabular}{ccc}
            &\text{Neutral affine gerbes}&\\
      &\text{over $k$}&
  \end{tabular}
}$};
            \node (A1_1) at (1, 1) {$\fbox{
  \begin{tabular}{ccc}
            &\text{Neutral Tannakian categories}&\\
            &\text{over $k$}&
  \end{tabular}
}$};
        
            \draw[implies-implies,double equal sign distance] (A0_0) -- node[anchor=south]{$\scriptstyle{}$} (A0_1);
            
            \draw[implies-implies,double equal sign distance] (A0_1) -- node[right] {$\scriptstyle{}$} (A1_1);
            \draw[implies-implies,double equal sign distance] (A0_0) -- node[anchor=south west]
            {$\scriptstyle{}$} (A1_1);
        \end{tikzpicture}
    \]
        In particular, the $2$-category of neutral gerbes and that of neutral
        Tannakian categories are essentially $1$-categories.
\end{Corollary}

The analogy among affine group schemes, affine gerbes and
    Tannakian categories is also reflected in terms of ``homomorphisms''. 
    \begin{Lemma}\label{relative gerbe = surjective} Let $f\colon \Gamma\to \Gamma'$ be a map of affine gerbes over
    $\Aff/k$. and let $f^*$ be the corresponding functor
    $\Vect(\Gamma')\to\Vect(\Gamma)$. The following are equivalent:
    \begin{enumerate}[label={\rm(\arabic*)}]
        \item The map $f$ is a relative gerbe;
        \item For some field extension $l/k$ and some section $x\in\Gamma(l)$,
            the map of affine $l$-group schemes $\Aut(x)\to\Aut(f(x))$ is
            faithfully flat;
        \item For every $k$-algebra $R$ and every section $x\in\Gamma(R)$ the
            map of affine $R$-group schemes $\Aut(x)\to\Aut(f(x))$ is
            faithfully flat;
        \item The functor $f^*$ is fully faithful and every
            subobject $W\subseteq f^*V'$ is the pullback of an object
            $W'\in\Vect(\Gamma')$.
    \end{enumerate}
\end{Lemma}
\begin{proof} For the equivalences (1)$\Leftrightarrow$(2)$\Leftrightarrow$(3),
    one just has to notice that being faithfully flat on the automorphism group
    schemes is an fpqc local property, so we are reduced to the case where $f$
    is a map of trivial gerbes. In this case, the equivalences are obvious.
    (3)$\Leftrightarrow$(4) follows from
    \cite[3.3.3]{Saavedra72}.
   One can also see it directly as follows. Full faithfulness of $f^*$ is equivalent to
   	$\cO_{\Gamma'}\cong f_*\cO_{\Gamma}$ and the condition on extension of subobjects
    satisfies the base change (as in the proof of { Lemma \ref{extending
    objects in general}}). 
So we can reduce the assertion to neutral
   	gerbes, where the result is classical and follows from \cite[Proposition 2.21]{Deligne-Milne1982}). 
\end{proof}

\section{Generalities on $F$-divided sheaves} \label{s:strat}
{
This section establishes the foundational definitions and key properties of
$F$-divided sheaves, setting the stage for the rest of the paper. The core
theme is understanding how these objects behave under various geometric
conditions, with a particular focus on the transition from regular to normal
algebraic stacks.

An $\FF_p$-algebraic space $X$ is called \emph{$F$-finite} if its absolute Frobenius map
$F_X$ is representable by a finite map of schemes. An $\FF_p$-algebraic stack $X$ is called
\emph{weakly $F$-finite} if it admits a smooth atlas by an $F$-finite scheme, i.e.,
there is a map $U\to X$ representable by a smooth fppf-covering of algebraic spaces,
where $U$ is an $F$-finite scheme. Clearly, any $F$-finite algebraic space is
weakly $F$-finite as an algebraic stack.

\subsection{$F$-divided sheaves are $F$-divided bundles}
}Let $X$ be { a locally Noetherian  $\FF_p$-algebraic stack}.
An \emph{$F$-divided (coherent) sheaf} on	$X$ is a sequence $\{E_i, \sigma_i\} _{i\in \ZZ_{\ge 0}}$ of coherent $\cO_{X}$-modules $E_i$ on $X$ and $\cO_X$-isomorphisms $\sigma _i:F_{X}^*E_{i+1}\to E_i$ of $\cO_{X}$-modules.
A morphism of $F$-divided sheaves  $\{E_i, \sigma_i\} \to  \{E'_i, \sigma'_i\} $
is a sequence of $\cO_X$-linear maps $\alpha_i: E_i\to E_i'$ such that
$\sigma_i'\circ F_X^*(\alpha_{i+1})=\alpha_i\circ \sigma_i$.
The category of $F$-divided sheaves on $X$ is denoted by ${\sf Coh} ^{\sf perf}(X)$ as it can be defined as 
$$\lim \left(... {\sf Coh} (X)\mathop{\rightarrow}^{F_X^*}{\sf Coh} (X)\mathop{\rightarrow}^{F_X^*}{\sf Coh} (X)\right).$$
Similarly, one can define the category of $F$-divided vector bundles  ${\sf Vect} ^{\sf perf}(X)$ as
$$\lim \left(... {\sf Vect} (X)\mathop{\rightarrow}^{F_X^*}{\sf Vect} (X)\mathop{\rightarrow}^{F_X^*}{\sf Vect} (X)\right).$$
We will also use an analogously defined category ${\sf QCoh} ^{\sf perf}(X)$ of quasi-coherent $F$-divided sheaves.

A \emph{unit} $\unit_X$ is the $F$-divided line bundle defined by the constant sequence $\{\cO_X\}_{i\in \ZZ_{\ge 0}}$ with canonical isomorphisms $F_{X}^*\cO_X\simeq \cO_X$.

The following fact is well-known (see \cite[Lemma 6]{dos_Santos2007} and
\cite[Proposition 1.3]{Bhatt}):

\begin{Proposition} \label{local-freeness-F-div-sheaves}
    { Let $X$ be a weakly $F$-finite $\FF_p$-algebraic stack}.
The inclusion ${\sf Vect} ^{\sf perf}(X)\subset {\sf Coh} ^{\sf perf}(X)$ is an equivalence of categories.
Equivalently, if  $\{E_i, \sigma_i\}$ is an $F$-divided sheaf then all $E_i$ are vector bundles.
\end{Proposition}
{
\begin{proof} We need to show that $E_i$ is finite locally free on $X$. Since
    $X$ is weakly $F$-finite, it admits a smooth atlas $u\colon Y\to X$, where $Y$ is an
    $F$-finite scheme. As $X$ is locally Noetherian, so is $Y$. By \cite[Prop.
    1.3]{Bhatt}, $u^*E_i$ is locally free. It follows that $E_i$ is locally
    free as well.
\end{proof}

 \subsection{The $h$-descent of $F$-divided sheaves}}
 
{
By Grothendieck's fpqc descent of quasi-coherent sheaves, it is obvious that
the fibered category $\Vect^\perf$ (or equivalently the 2-functor
$\Vect^\perf(-)$) is a stack in the fpqc-topology.  However, it is less obvious
that it also satisfies $h$-descent.
The following result is due to B. Bhatt \cite[Theorem
3.2]{Bhatt}, who  kindly allowed us to include it into this paper.
The proof is based on $h$-descent of vector bundles on locally Noetherian derived schemes due to D. Halpern-Leistner and A. Preygel (see \cite{Halpern-Leistner-Preygel2023}).}

\medskip

\begin{Theorem} [\emph{B. Bhatt}]	\label{h-descent-for-F-divided}
    { The functor $ {\sf Vect} ^{\sf perf}(-)$ satisfies descent
     for $h$-coverings of locally Noetherian}  $\FF_p$-schemes.
\end{Theorem}
{
\begin{proof}
We first extend $\mathrm{Vect}^{\perf}(-)$ from classical schemes to derived schemes (modeled by simplicial commutative $\FF_p$-algebras, i.e.,\ animated $\FF_p$-algebras). For a derived scheme $X$, we define
\[
\mathrm{Vect}^{\perf}(X) := \lim \Bigl( \cdots \xrightarrow{F^*} \mathrm{Vect}(X) \xrightarrow{F^*} \mathrm{Vect}(X) \Bigr),
\]
where $F$ is the absolute Frobenius and $\mathrm{Vect}(X)$ denotes the $\infty$-groupoid of vector bundles on $X$. Concretely, an object of $\mathrm{Vect}^{\perf}(X)$ is a compatible sequence of vector bundles $\{E_n\}$ on $X$ together with isomorphisms
\[
\varphi_n : F^*(E_{n+1}) \xrightarrow{\sim} E_n.
\]

By a theorem of Halpern-Leistner and Preygel \cite{Halpern-Leistner-Preygel2023}, the functor 
$\mathrm{Vect}(-)$ is an \(h\)-sheaf on locally Noetherian derived stacks.
 Since $\mathrm{Vect}^{\perf}(-)$ is defined as a homotopy limit of copies of $\mathrm{Vect}(-)$ along Frobenius, it follows formally that $\mathrm{Vect}^{\perf}(-)$ also satisfies $h$-descent on derived schemes.

It remains to compare the derived and classical situations. Let $X$ be a
locally Noetherian derived $\FF_p$-scheme with classical truncation $X_0$. We claim that
\[
\mathrm{Vect}^{\perf}(X) \simeq \mathrm{Vect}^{\perf}(X_0).
\]
At first glance, there is a mismatch: the left-hand side is an $\infty$-groupoid (moduli of $F$-divided bundles up to isomorphism), while the right-hand side is usually defined as a category of $F$-divided sheaves with all morphisms. However, the equivalence can be understood at the groupoid level, and in fact also holds at the categorical level. The latter follows by analyzing morphisms: if $A,B$ are vector bundles, then
\[
\underline{\mathrm{Hom}}(A,B) \cong A^\vee \otimes B,
\]
and this tensor construction is compatible with passage from $X$ to $X_0$ by Zariski descent.

To check the claim affine locally, let $R$ be an animated $\FF_p$-algebra. Suppose $\{M_n\} \in \lim_F \mathrm{Vect}(R)$ is an $F$-divided vector bundle with image $\{Q_n\} \in \lim_F \mathrm{Vect}(\pi_0(R))$. Then the natural map
\[
\mathbf{R}\!\lim_n M_n \;\;\longrightarrow\;\; \mathbf{R}\!\lim_n Q_n
\]
is an isomorphism. Writing $R \simeq \varprojlim R_m$ as an inverse limit of
its Postnikov truncations $R_m = \tau_{\leq m}R$, it suffices to prove the
analogous statement for truncated rings $R_m$. For truncated $R_m$, the
Frobenius endomorphism factors through the projection $R_{m+1} \to R_m$. This
Frobenius factorization property (see \cite[§11, proof of
Thm.~11.6]{Bhatt-Scholze2017}) shows that the $F$-divided structure depends only on $\pi_0(R)$. Consequently, the derived and classical limits agree.

Thus $\mathrm{Vect}^{\perf}(X) \simeq \mathrm{Vect}^{\perf}(X_0)$ for any derived $X$, and since $\mathrm{Vect}^{\perf}(-)$ satisfies $h$-descent on derived schemes, it also satisfies $h$-descent on classical $\FF_p$-schemes.
\end{proof}
}

Let us also note the following generalization of  {\cite[Proposition 3.3,
(ii)]{Esnault-Srinivas2016}:}

\begin{Lemma}\label{passing-to-universal-homeomorphism}
    Let $f\colon Y\to X$ be a  finite
    universal homeomorphism of { locally Noetherian $\FF_p$-algebraic stacks. Then
    the pullback by $f$ induces an}
    equivalence of categories ${\sf Vect} ^{\sf perf}(X)\simeq {\sf Vect} ^{\sf
    perf}(Y)$. 
\end{Lemma}

\begin{proof} { The map $f$ induces a functor $\Hom_{\FF_p}(X,\Vect^\perf)\to
    \Hom_{\FF_p}(Y,\Vect^\perf)$, and we have to show that it is an
    equivalence. Choose any groupoid presentation $R\rightrightarrows U$ of $X$
    by algebraic spaces. Using the fact that $\Vect^\perf$ is a stack in the
    fppf-topology, we are reduced to the case where $X$ is an algebraic space,
    and similarly, we can further assume that $X$ is a scheme. As $f$ is
finite, $Y$ is also a scheme.}

        In light of \cite[Definition
        \href{https://stacks.math.columbia.edu/tag/0ETS}{0ETS}]{StacksProject}, the morphism
        $f$ is an $h$-covering. The
        fact that the (higher) diagonals of $f$ are closed immersions of finite
        presentation implies
        that they are also $h$-coverings and monomorphisms.
        We now apply Lemma \ref{when the pullback in a stack is an equivalence}
        to $\cX\coloneqq\Vect^\perf(-)$ and the category  $\cC$ of
        $\FF_p$-schemes equipped with the $h$-topology.
\end{proof}

{
\begin{Lemma}\label{when the pullback in a stack is an equivalence} Let $\cC$ be a site, and let $f\colon Y\to X$ be a covering in
    $\cC$ such that the diagonals $\Delta\colon Y\to Y\times_XY$ and
    $\Delta_2\colon Y\to Y\times_XY\times_XY$ are coverings and monomorphisms.
    Then for any stack $\cX$ on $\cC$, the pullback functor $\cX(X)\to\cX(Y)$
    is an equivalence.
\end{Lemma}
\begin{proof} Since $\Delta$ and $\Delta_2$ are monomorphisms, their diagonals are isomorphisms. Given that they are also coverings and that $\cX$ is a stack, it follows that both induce equivalences of categories:
    \begin{equation}\label{pullback along the diagonals are equivalences}
\Delta^*\colon\cX(Y \times_X Y) \simeq \cX(Y) \quad \text{and} \quad \Delta_2^*\colon\cX(Y \times_X Y \times_X Y) \simeq \cX(Y).
\end{equation}
    Since $\cX$ is a stack for the topology on $\cC$, it satisfies descent
    along the covering $f : Y \to X$. Consequently, the category $\cX(X)$ is
    equivalent to the category $\DD(f)$ of descent data relative to $f$. An
    object of $\DD(f)$ is a pair $(E, \phi)$ where $E \in \cX(Y)$ and $\phi:
    p_1^*E \xrightarrow{\sim} p_2^*E$ is an isomorphism in $\cX(Y \times_X Y)$
    satisfying the cocycle condition {
\begin{equation}\label{cocycle condition} 
        p_{23}^*(\phi) \circ p_{12}^*(\phi) = p_{13}^*(\phi)
        \end{equation}in
    $\cX(Y \times_X Y \times_X Y)$. 
Applying $\Delta_2^*$ to \eqref{cocycle
condition}, we get \(\Delta_2^*p_{23}^*(\phi) \circ \Delta_2^*p_{12}^*(\phi) =
\Delta_2^*p_{13}^*(\phi)\). This is nothing but
$\Delta^*(\phi)\circ\Delta^*(\phi)=\Delta^*(\phi)$. Thus
$\Delta^*(\phi)=\id_E$.}

Combined with the earlier result
\eqref{pullback along the diagonals are equivalences} -- which, by the definition
of $\Delta$, states that pullback along $\Delta$ induces an equivalence $\cX(Y
\times_X Y) \simeq \cX(Y)$ -- this implies that the isomorphism $\phi$ is uniquely determined. Therefore, every object $E \in \cX(Y)$ admits a descent datum, and this datum is unique up to unique isomorphism.
It follows that the forgetful functor $\DD(f) \to \cX(Y)$, which sends $(E, \phi)$ to $E$, is an equivalence of categories. Since $\cX(X) \simeq \DD(f)$, we conclude that the pullback functor $\cX(X) \to \cX(Y)$ is also an equivalence.
\end{proof}
}

As a special case we get the following corollary 
(see also \cite[Proposition 1.5]{Gieseker1975} or \cite[Lemma 1.1]{Bhatt}).
{ Recall a closed immersion $Y\subset X$ is called a thickening if it is a
surjective
(cf.~\cite[\href{https://stacks.math.columbia.edu/tag/04EW}{04EW}]{StacksProject}).}

\begin{Lemma} \label{lem:finiteness}
    If { $Y\subset X$ is a  thickening} of { locally} Noetherian $\FF_p$-schemes
   then the restriction gives rise to an equivalence of categories 
	${\sf Vect} ^{\sf perf}(X)\simeq {\sf Vect} ^{\sf perf}(Y)$.
\end{Lemma}

%\medskip

{ \subsection{Reflexive $F$-divided sheaves on normal algebraic stacks}}
Let us assume now that $X$ is a normal locally Noetherian algebraic stack. { Note that since
$X$ is clearly locally irreducible, i.e. every point admits an open irreducible
neighbourhood, if it is connected, then it is irreducible.}
Let  ${\sf Ref} (X)$ denote the category of coherent reflexive $\cO_X$-modules. It is a full subcategory of $\Coh(X)$ and   the inclusion functor ${\sf Ref} (X)\subset {\sf Coh} (X)$ comes with a left adjoint $(\cdot )^{**}:{\sf Coh} (X)\to {\sf Ref} (X)$ given by the reflexive hull.
The composition
$${\sf Ref} (X)\subset{\sf Coh} (X) \mathop{\rightarrow}^{F_X^*}{\sf Coh} (X)\mathop{\rightarrow}^{(\cdot )^{**}}{\sf Ref} (X)$$
is denoted by $F_X^{[*]}$. 
The composition
$${\sf Ref} (X)\subset{\sf Coh} (X) \mathop{\rightarrow}^{(F_X^n)^*}{\sf Coh} (X)\mathop{\rightarrow}^{(\cdot )^{**}}{\sf Ref} (X)$$
is denoted by $F_X^{[n]}$.
As above we define the category of $F$-divided reflexive sheaves  ${\sf Ref} ^{\sf perf}(X)$  as 
$$\lim \left(... {\sf Ref} (X)\mathop{\rightarrow}^{F_X^{[*]}}{\sf Ref} (X)\mathop{\rightarrow}^{F_X^{[*]}}{\sf Ref} (X)\right).$$
If $X$ is regular then $F_X$ is flat and $F_X^{[*]}: {\sf Ref} (X)\to {\sf Ref} (X)$ is 
the restriction of $F_X^*: {\sf Coh} (X)\to {\sf Coh} (X)$. In this case ${\sf Ref} ^{\sf perf}(X)$ is a subcategory of $ {\sf Coh} ^{\sf perf}(X)$ and hence by Proposition \ref{local-freeness-F-div-sheaves} it is equivalent to  $ {\sf Coh} ^{\sf perf}(X)$.
In general, since ${\sf Vect} (X)$ is a subcategory of ${\sf Ref} (X)$ and the restriction of $F_X^{[*]}$ to ${\sf Vect} (X)$ coincides with $F_X^{*}$,   ${\sf Vect} ^{\sf perf}(X)$ is a subcategory of ${\sf Ref} ^{\sf perf}(X)$.
As seen below these categories are in general not equivalent.

Let us recall that by Kunz's theorem any { locally} Noetherian $F$-finite $\FF_p$-scheme $Y$ is excellent (see \cite[Theorem 10.5]{Ma-Polstra}). In particular, the regular locus $Y_{\sf reg}$ of $Y$ is open in $Y$. 
{ For such a scheme $Y$,  by \cite[\href{https://stacks.math.columbia.edu/tag/0EBJ}{0EBJ}]{StacksProject} the adjoint pair $(j^*, j_*)$ induces mutually quasi-inverse equivalences between ${\sf Ref}(Y)$ and ${\sf Ref}(Y_{\sf reg})$, where $j\colon Y_{\sf reg} \hookrightarrow Y$ denotes the inclusion. 

This result extends to normal locally Noetherian weakly $F$-finite algebraic stacks. Let $Y$ be such a stack and let $u\colon U \twoheadrightarrow Y$ be a smooth atlas with $U$ a locally Noetherian $F$-finite scheme. We define the regular locus of $Y$ by $Y_{\sf reg} \coloneqq u(U_{\sf reg})$, where $U_{\sf reg}$ is the (open) regular locus of $U$; note that $U_{\sf reg} = u^{-1}(Y_{\sf reg})$ and this definition is independent of the choice of atlas.

The equivalence for stacks follows from the scheme case because:
(1) $u^*$ commutes with sheaf Hom, so a coherent sheaf $E$ on $Y$ is reflexive if and only if $u^*E$ is reflexive on $U$;
(2) Cohomology commutes with base change along the smooth morphism $u$.
Therefore, the adjoint pair $(j^*, j_*)$ induces mutually quasi-inverse equivalences between ${\sf Ref}(Y)$ and ${\sf Ref}(Y_{\sf reg})$.
gives the stacky version as you complained about Frobenius on stacks.}

\begin{Lemma}\label{description-of-Ref-perf} { Let $X$ be a locally
    Noetherian normal weakly $F$-finite $\FF_p$-algebraic stack}
 and let $j\colon X_{\sf reg}\subset X$ denote the canonical open embedding. Then the restriction $j^*$ defines an equivalence of categories
${\sf Ref} ^{\sf perf}(X)\to {\sf Vect} ^{\sf perf}(X_{\sf reg})$.  In
particular, $\Reff^\perf(X)$ is an abelian category.
\end{Lemma}

\begin{proof}
Note that in general restriction of a reflexive  $\cO_X$-module to $X_{\sf reg}$ need not be a vector bundle.
But it is certainly a coherent $\cO_{X_{\sf reg}}$-module and we have a well-defined functor 
$j^*: {\sf Ref} (X)\to {\sf Coh} (X_{\sf reg})$. Since the Frobenius morphism is flat on $X_{\sf reg}$, we have a commutative diagram
$$
\xymatrix{
{\sf Ref} (X) \ar[r]^{F_{X}^{[*]}} \ar[d]^{j^*}&{\sf Ref} (X)\ar[d]^{j^*}\\
{\sf Coh} (X_{\sf reg}) \ar[r]^{F_X^*}&{\sf Coh} (X_{\sf reg})}
$$
inducing the functor ${\sf Ref} ^{\sf perf}(X)\to {\sf Coh} ^{\sf perf}(X_{\sf reg})\simeq {\sf Vect} ^{\sf perf}(X_{\sf reg})$. To obtain a quasi-inverse note that $j_*$ defines a functor
${\sf Vect} (X_{\sf reg})\to {\sf Ref} (X)$. We also have a commutative diagram
$$
\xymatrix{
	{\sf Ref} (X) \ar[r]^{F_{X}^{[*]}} &{\sf Ref} (X)\\
	{\sf Vect} (X_{\sf reg}) \ar[r]^{F_X^*}\ar[u]^{j_*}  &{\sf Vect} (X_{\sf reg})
	\ar[u]^{j_*} 
	}
$$
inducing the functor  $j_*: {\sf Vect} ^{\sf perf}(X_{\sf reg}) \to {\sf Ref}
^{\sf perf}(X)$, which is the required quasi-inverse.
\end{proof}

{
\subsection{Tannakian property of $F$-divided sheaves}
}

\begin{Definition}
	Let $R$ be an $\FF_p$-algebra. The \emph{inverse limit perfection} is the inverse limit
	\[R^{\sf perf}=\lim \left(... \stackrel{F_R}{\longrightarrow}R
	\stackrel{F_R}{\longrightarrow}R
	\stackrel{F_R}{\longrightarrow}R\right)\] over the Frobenius maps.
\end{Definition}
This ring is clearly perfect.
If $R$ is reduced then $R^{\sf perf}=\bigcap _{m\ge 0}R^{p^m}$.

{

    \begin{Theorem}\label{F-divided-sheaves-form-Tannakian-category}
        Let $X$ be a weakly $F$-finite locally Noetherian connected {
            $\FF_p$-}algebraic
            stack. Then 
        \begin{enumerate}[label={\rm(\Alph*)}]
            \item $\Vect^\perf(X)=\Coh^\perf(X)$;
            \item { for any map $u\colon T\to X$, where $T$ is a nonempty scheme,
                the pullback functor $\Vect^\perf(X)\to \Vect (T)$,
                $\{E_i,\sigma_i\}_{i\in\NN}\mapsto u^*E_0$ is
            faithful;}
            \item $\End(\mathbbm{1}_X)=\cO_X(X)^\perf
                %=\Set{a\in\cO_X(X)|\FF_p[a]\text{ is étale
                %over $\FF_p$}}
                $ is a field, and it is the maximal perfect subfield contained
                in $\cO_X(X)$;
            \item $\Vect^\perf(X)$ is a Tannakian category over
                $\End(\mathbbm{1}_X)$;
            \item The Tannakian gerbe $\Pi_X^{\Fdiv}$ corresponding to $\Vect^\perf(X)$, is pro-smooth
                banded.
        \end{enumerate}
    \end{Theorem}
    \begin{proof} %Clearly, if $X$ is $F$-finite, then all of its local rings
        %are $F$-finite, and so are the residue fields.
        (A) is already Proposition \ref{local-freeness-F-div-sheaves}. Let $f=\{f_i\}_{i\in\ZZ_{\ge 0}}\colon \EE=\{E_i, \sigma _i\}_{i\in\NN}\to \FF=\{F_i, \tau_i\}_{i\in\NN}$ be a map in $\Coh^\perf(X)$. 
     Since the Frobenius pullback is right exact we can define the cokernel of $f$ by setting
        $\coker(f):=\{ \coker (f_i), \tau_i'\}_{i\in \ZZ_{\ge 0}}$, where
        isomorphisms $\tau_i'$ induced from $\tau_i$. Since each  $\coker
        (f_i)$   is a vector bundle by (A), we can also define the image of $f$
        as $\im(f)\coloneqq \{ \im(f_i), \tau_i|_{\im(f_i)}      \}
        _{i\in\ZZ_{\ge 0}}$.  This makes sense as $\im (f_i) =\ker(F_i\to \coker (f_i))$ is locally free. Now as each $\im(f_i)$ is locally free, we can construct the kernel of $f$ as $\ker(f)\coloneqq \{ \ker(f_i), \sigma_i|_{\ker(f_i)} \}_{i\in\ZZ_{\ge 0}}$.  
       This makes sense as $\ker(f_i) =\ker(E_i\to \im(f_i))$ is locally free. Therefore
        $\Coh^\perf(X)$ is abelian.

        Let $f\colon \EE\to \FF$ be a map in
        $\Vect^\perf(X)$. If the  pullback of $f$ to $T$ is the zero map, then $\im(f)$ is of rank 0,
        i.e. $f=0$. This proves (B). Suppose $\EE=\FF=\mathbbm{1}_X$. If $f\neq 0$,
        then $\ker(f), \coker(f)$ must of rank 0, so $f$ is an isomorphism.
        This shows that $k:=\End(\mathbbm{1}_X)$ is a field. From the very definition $k$ is the inverse limit
        perfection of $\cO_X(X)$. Hence $k$ is a perfect subfield of
        $\cO_X(X)$. Conversely, if $l\subseteq \cO_X(X)$ is a perfect subfield,
        then the pullback functor $\Coh(\Spec(l))=\Coh^\perf(\Spec(l))\to\Coh^\perf(X)$ is
        faithful, so $l=\End(\mathbbm{1}_l)\subseteq\End(\mathbbm{1})=k$, i.e.,
        $k$ is maximal. This yields (C).

{
          From (A) we see that $\Coh^\perf(X)$ is a rigid abelian tensor
          category. To prove (D), we just have to find a fiber functor. Since
          $X\neq\emptyset$, there is a section $s\in X(T)$ (or equivalently a
          1-morphism $s\colon T\to X$), where $T$ is a nonempty 
          affine scheme. Then the pullback $s^*\colon
          \Vect^\perf(X)\to\Vect(T)$ is $k$-linear, tensorial, exact and
          faithful -- hence a fiber functor.

        }

        (E) follows from \cite[Theorem 6.23
        (1)]{TZ1}. 
    \end{proof}

\begin{Remark}
	(D) in the above theorem generalizes \cite[Proposition 3.3 (i)]{Esnault-Srinivas2016}. We give a different proof of this result.
\end{Remark}

{
    \section{Local behavior and generic restrictions}\label{Local behavior
regular}
We now study the local properties of \(F\)-divided bundles over normal and
regular schemes, focusing on generic points. The main results of this section
establish ``extension of subobjects'' and ``full faithfulness'' for restriction
functors, showing that the global structure of an \(F\)-divided bundle is
frequently controlled by its restriction to a dense open set or the generic
point of a regular scheme.
}

\begin{Lemma}\label{perfection-for-normal-rings}
	Let $R$ be a Noetherian  geometrically unibranch integral $\FF_p$-algebra, and let $K$ denote the
	field of fractions of $R$. Then the	canonical map $R^{\sf perf}\to K^{\sf perf}$ is an isomorphism.
\end{Lemma}
\begin{proof} By \cite[Lemma \href{http://stacks.math.columbia.edu/tag/035R}{035R}
		and Lemma
		\href{http://stacks.math.columbia.edu/tag/0GIQ}{0GIQ}]{StacksProject}
		the normalization morphism $X^{\nu}\to X\coloneqq\Spec(R)$ is a finite
        universal homeomorphism.  Set $X^{\nu}=\Spec(\bar{R})$. 
		By Lemma \ref{passing-to-universal-homeomorphism}, $R^\perf=\bar{R}^\perf$,
    so we can assume that $R$ is normal. 
	Let $k\coloneqq K^\perf$ be the maximal perfect subfield of $K$.
	It is enough to show that $k\subseteq R$.
	Since $R$ is satisfies $S_2$, we have
	$R=\displaystyle\bigcap_{\textup{ht}(\pp)=1}R_\pp$, where $\pp$ runs over
	all height 1 prime ideals of $R$ (cf.~
	\cite[\href{https://stacks.math.columbia.edu/tag/031T}{Lemma 031T}]{StacksProject}). Replacing $R$ by $R_\pp$, we may assume 
	that $R$ is a DVR. If $a\in k\subseteq K$, then the valuation of $a$ must
	be 0, because it is infinitely $p$-divisible. Thus $a\in R$ as desired.
\end{proof}

\begin{Corollary}\label{endomorphism-sheaf-of-trivial-divided-line-bundle}
	Let $X$ be a  Noetherian geometrically unibranch integral $\FF_p$-scheme. Then
	the sheaf $\cEnd \unit_X$ of endomorphisms of $\unit_X$ in ${\sf Vect} ^{\sf perf}(X)$ is the constant sheaf associated to $\cO_X(X)^{\perf}$. Moreover,  $\cO_X(X)^{\perf}$ is canonically isomorphic to the inverse limit perfection $K(X)^{\perf}$ of the function field of $X$.
\end{Corollary}

\begin{proof}
	The corollary follows immediately from Lemma \ref{perfection-for-normal-rings} and the fact that for any open $U\subset X$ we have $(\cEnd \unit_X) (U)=\cO_X(U)^{\perf}$.	
\end{proof}

\begin{Lemma}\label{extending objects in the regular case} Let $X$ be a
    Noetherian regular connected $F$-finite $\FF_p$-scheme and
	let $\eta : \Spec(K(X))\hookrightarrow X$ be the generic point. If
	$\EE\in\Vect^\perf(X)$, and if $\GG_\eta\subseteq \eta^*\EE$ is a subobject in
	$\Vect^\perf(\eta)$, then there exists a subobject $\GG\subseteq \EE\in\Vect^\perf(X)$ such that
	$\eta^*\GG=\GG_\eta$ as subobjects of $\eta^*\EE$. 
\end{Lemma}
\begin{proof} 
Since the Frobenius map $F_X$ is flat, the category ${\sf QCoh} ^{\sf perf}(X)$ is abelian. Moreover,
$F_X^*$ commutes with $\eta_*$ and hence $\eta_*$ induces a well defined
functor  ${\sf QCoh} ^{\sf perf}(\eta)\to {\sf QCoh} ^{\sf perf}(X)$.
So in the category ${\sf QCoh} ^{\sf perf}(X)$ we can define the subobject
$\GG$ of $\EE$ as $\eta_*\GG_\eta\times_{\eta_*\eta^*\EE} \EE$. Since $X$ is
Noetherian, $\GG$ lies in ${\sf Coh} ^{\sf perf}(X)$ and hence by Proposition
\ref{local-freeness-F-div-sheaves} also in  $\Vect^\perf(X)$.
\end{proof}

Let $\Pic _F(X)$ denotes the group of isomorphism classes of $F$-divided line bundles on $X$. 

\begin{Proposition}\label{injectivity-on-F-divided-Pic}
Let $X$ be a Noetherian normal integral $\FF_p$-scheme and let $\eta $ be the generic point of $X$. Then the restriction map $\Pic _F(X)\to \Pic _F(\eta)$ is injective. In particular, for any non-empty open subset $U\subset X$  the restriction map $\Pic _F(X)\to \Pic _F(U)$ is injective.
\end{Proposition}

\begin{proof}
Let $\LL$ be an $F$-divided line bundle in $\Vect^\perf(X)$. We need to show that  an isomorphism $\alpha: \mathbbm{1}_{\eta}\cong\LL _{\eta}$ extends to an isomorphism $\mathbbm{1}_X\cong \LL$.
%Since by Theorem \ref{F-divided-sheaves-form-Tannakian-category} the restriction $\Vect^\perf(X)\to \Vect (\eta)$ is faithful, 
Since $X$ is integral, the restriction $\Vect(X)\to \Vect (\eta)$ is faithful, and hence $\Vect^\perf(X)\to \Vect ^{\perf}(\eta)$ is also faithful. So an extension of
$\alpha$, if it exists, is unique.  Hence it is sufficient to show that for
each affine open $V\subset X$ the isomorphism $\alpha$ extends (uniquely) to an
isomorphism  $\mathbbm{1}_V\cong \LL_V$. Thus in the following we can assume that $X=\Spec R$ is affine.

Let us write $\LL=\{L_i, \sigma_i\}$, where $L_i$ are projective  $R$-modules of rank $1$ and $\sigma _i:F_{R}^*L_{i+1}\to L_i$ are isomorphisms of $R$-modules. 
Let $K=\kappa (\eta)$ denote the function field of $X$. A morphism
$\mathbbm{1}_X \to \LL_X$  can be viewed as a sequence
$(l_i)_{i\in\NN}$ with $l_i\in L_i$ such that $\sigma_i(F_R^*l_{i+1})=l_i$.
Similarly the isomorphism $\alpha\colon\mathbbm{1}_\eta \to \LL_\eta$ can be
seen as a sequence  $(l_i)_{i\in\NN}$ with $l_i\in L_i\otimes_RK$ such that
$\sigma_{i,K}(F_K^*l_{i+1})=l_i$, where
$\sigma_{i,K}\coloneqq\sigma_i\otimes_RK$. To show that $\alpha$ extends to
$X$, it is enough to show that for all $i\ge 0$, $l_i$ lies in $L_i\subseteq L_i\otimes_RK$. 

Since $R$ is normal, each $L_i$ is a projective $R$-module of rank $1$, and $\Hom(-,-)$ commutes with limits in the second variable, we have
\begin{equation*}
\resizebox{0.9\textwidth}{!}{
\(\displaystyle L_i=L_i\otimes_RR=  L_i\otimes _R\left( \bigcap_{\textup{ht}(\pp)=1}R_\pp
\right)=\Hom _R\left( L_i^*,  \bigcap_{\textup{ht}(\pp)=1}R_\pp
\right)= \bigcap_{\textup{ht}(\pp)=1}\Hom _R(L_i^*, R_\pp ) =
\bigcap_{\textup{ht}(\pp)=1}L_i\otimes _R R_\pp  \hookrightarrow
L_i\otimes _RK,\)}
\end{equation*}
where $L_i^*\coloneqq\Hom _R(L_i,R)$.
Therefore it is sufficient to show that for every height $1$ prime ideal $\pp$
of $R$, $l_i$ lies in $L_i\otimes_RR_\pp$. So in the following we can assume
that $R$ is a DVR with a discrete valuation $v$, and we choose an $R$-module isomorphism
$L_i\cong R$ for each $i$.
In this case, $\sigma_i$  can be viewed as the multiplication by a unit $u_i\in R^*$.
Now the equality $\sigma_{i,K}(F_K^*l_{i+1})=l_i$ reads $u_i(l_{i+1}^p)=l_i$ in $K$.
This implies that $v(l_i)=pv(l_{i+1})$, hence $v(l_i)=p^nv(l_{i+n})$. This
shows that $v(l_i)=0$ for all $i$, or equivalently, $l_i\in R^*\subseteq R\cong
L_i$, as desired.  
\end{proof}

\begin{Remark}
Note that the above proof shows that an isomorphism $\alpha: \mathbbm{1}_{\eta} \cong  \LL _{\eta}$ extends to a unique isomorphism $\mathbbm{1}_{X}\cong \LL$. In particular, for $\LL=\unit_X$ this implies that the map $\End \unit_X\to \End \unit _{\eta}=K^{\sf perf }$ is an isomorphism. This gives another proof of  Lemma \ref{perfection-for-normal-rings} and Corollary \ref{endomorphism-sheaf-of-trivial-divided-line-bundle}. 	
\end{Remark}

The following lemma generalizes \cite[Lemma 2.5]{Kindler2015}. Note that due to  lack of local coordinates in our set-up,
the proof from \cite{Kindler2015} does not work  and we need a different approach (see however Remark \ref{Kindler-works-nonetheless}).

\begin{Lemma}\label{full-faithfullness-of-restriction-to-open-regular}
Let $X$ be an integral Noetherian regular $F$-finite $\FF_p$-scheme and let
$\eta$ be the generic point of $X$. Then the restriction
functor ${\sf Vect} ^{\sf perf}(X)\to {\sf Vect} ^{\sf perf}(\eta)$ is fully faithful.
\end{Lemma}

    \begin{proof} 
 We need to show that for all $\EE_1, \EE_2\in{\sf Vect} ^{\sf perf}(X) $ the restriction map
  $$\eta^*: \Hom _X (\EE_1, \EE_2)\to \Hom _{\eta} (\eta^*\EE_1, \eta^* \EE_2)$$
  is an isomorphism. This is equivalent to saying that for any $\EE\in{\sf Vect} ^{\sf perf}(X)$ 
 the restriction map defines an isomorphism between $\Hom_X(\unit_X, \EE )$   	
  and  $\Hom_\eta(\unit_\eta, \eta ^*\EE )$. Thanks to Lemma \ref{extending objects in the regular case},
        we are reduced to showing that if $\LL$ is a line bundle in
        $\Vect^\perf(X)$ whose restriction $\eta^*\LL\cong \mathbbm{1}_\eta$,
        then $\LL\cong \mathbbm{1}_X$. This follows from Proposition \ref{injectivity-on-F-divided-Pic}.
\end{proof}

\begin{Corollary}\label{Kindler}
Let $X$ be an integral Noetherian regular $F$-finite $\FF_p$-scheme, and let
 $\eta$ be the generic point of $X$. Then the induced $1$-morphism
$\Pi_\eta^{\Fdiv}\to \Pi_X^{\Fdiv}$ is a relative gerbe over
$K(X)^\perf$. In particular, for any dense open $U\subseteq X$, the
$1$-morphism $\Pi_U^{\Fdiv}\to \Pi_X^{\Fdiv}$ is a relative gerbe over
$K(X)^\perf$.
\end{Corollary}

\begin{proof} This follows from Corollary \ref{endomorphism-sheaf-of-trivial-divided-line-bundle}, Lemma \ref{extending objects in the
        regular case}, Lemma \ref{full-faithfullness-of-restriction-to-open-regular}
        and Lemma \ref{relative gerbe = surjective}.
\end{proof}

\section{$D$-modules on $F$-finite schemes}

A large part of the theory related to differential operators and Cartier's descent is worked out for schemes that are smooth over an algebraically closed field. Here we extend this theory to regular $F$-finite schemes.

Let $R$ be a Noetherian  ring of prime characteristic $p>0$. 
We say that a finite set  $\{r_1,...,r_n\}$ of  elements of $R$ is a \emph{$p$-basis of $R$ (over $R^p$)} if they generate $R$ as a ring over $R^p$  and the monomials $\{r_1^{i_1}...r_n^{i_n}\}_{0\le i_j<p}$ are linearly independent over $R^p$. Then it is easy to see that $\Omega_R= \Omega_{R/\FF_p}\simeq \Omega_{R/R^p}$ is a free $R$-module with basis $\{ dr_1,...,dr_n\}$. If $R$ has a $p$-basis and it is reduced then by \cite[Theorem 2]{Tyc1988} $R$ is formally smooth over $\FF_p$. In particular, by \cite[\href{https://stacks.math.columbia.edu/tag/0H7U}{Theorem 0H7U}]{StacksProject} $R$ is regular.
By \cite[Corollary 3.2]{Kimura-Niitsuma1982}  any  Noetherian regular $F$-finite local $\FF_p$-algebra has a $p$-basis (see also the proof of Proposition \ref{existence-of-local-p-coordinates} for a different proof). 

\begin{Proposition}\label{Tyc}
	Let $R$ be a Noetherian  $\FF_p$-algebra. If there exist $r_1,...,r_n\in R$ such that
	$dr_1,...,dr_n$ is an $R$-basis of $\Omega_R$, then $R$ is $F$-finite and $\{r_1,...,r_n\}$ is a $p$-basis of $R$.
	If moreover $R$ is reduced then $R$ is regular.
\end{Proposition}

\begin{proof}
	$F$-finiteness of $R$ follows from \cite[Proposition 1]{Fogarty1980}. The fact that  $\{r_1,...,r_n\}$ is a $p$-basis of $R$ follows from the more general \cite[Theorem 1]{Tyc1988}. The last part follows from Kunz's theorem (see, e.g., \cite[Theorem 1.1]{Ma-Polstra}).
\end{proof}

	Let $R$ be a Noetherian $F$-finite ring of prime characteristic $p>0$. By \cite[Theorem 10.9]{Ma-Polstra} $R$
	is a homomorphic image of a Noetherian $F$-finite regular ring of finite Krull dimension (note that there exist  Noetherian regular rings of infinite Krull dimension). In particular, $R$ has a finite Krull dimension.

	The following proposition is an analogue of existence of a system of local coordinates for smooth morphisms
	(see also \cite[Corollary 5.6]{Fink2025}).
	
	\begin{Proposition}\label{existence-of-local-p-coordinates}
		Let $X$ be a connected Noetherian regular $F$-finite $\FF_p$-scheme and let $x$ be a point of $X$. Let $n$ be the rank of the $\cO_X$-module $\Omega_X$ and let $k=\End (\unit _X)$.
		Then there exists an open neighbourhood $U$ of $x$ and a formally \'etale $k$-morphism $f:U\to \AA^n_{k}$.
		Moreover, we have equality
		$$n=\dim \cO_{X,x}+(\kappa (x): \kappa (x)^p)_p,$$
		where $(\kappa (x): \kappa (x)^p)_p=\log _p \dim _{\kappa (x)^p}(\kappa (x))$ is the $p$-degree of $\kappa (x)/\kappa (x)^p$.
	\end{Proposition}
	
	\begin{proof}   
		We have a standard short exact sequence
		$$m_x/m_x^2\to \Omega_{X,x}\otimes \kappa (x)\to \Omega_{\kappa (x)}\to 0$$
		of $\kappa (x)$-modules. By \cite[Satz 1]{Berger-Kunz1961} (see also \cite[Theorem 6.7]{Kunz-book} for a modern formulation) this sequence is also left exact.	
		Let us choose $r_1,...,r_s\in m_x$ so that its classes form a $\kappa (x)$-basis of  $m_x/m_x^2$ (equivalently, $r_1,...,r_s$ form a minimal set of generators of $m_x$). Let us also choose elements $r_{s+1},..., r_n\in \cO_{X,x}$ such that $d\bar{r}_{s+1},..., d\bar{r}_n$ form a $\kappa (x)$-basis of $\Omega_{\kappa (x)}$. Here $\bar{r_i}$ denotes the class of $r_i$ in $\kappa (x)$. Then $dr_1,...,dr_n$ form a $\kappa (x)$-basis of $\Omega_{X,x}\otimes \kappa (x)$. Since $\Omega_{X,x}$ is a free $\cO_{X,x}$-module, by Nakayama's lemma $dr_1,...,dr_n$
		are its free generators. Then $r_1,...,r_n$ extend to sections of $\cO_{X}$ such that $dr_1,...,dr_n$ generate $\Omega_X$ in some open neighbourhood $U\subset X$. We claim that these sections define the required morphism. 
		To check this we can assume that $U=\Spec R$ is affine. Then $k= R^{\sf
			perf}\subset R$ and the homomorphism $\varphi: A=k[x_1,...,x_n]\to R$ mapping
		$x_i$ to $r_i$ is $k$-linear.  Note that by  Proposition \ref{Tyc} the elements
		$\{r_1,...,r_n\}$ form a $p$-basis of $R$.  Since $R$ is regular,
        \cite[Suppl\'ement, Th\'eor\`eme  30]{Andre1974}
        implies that the cotangent complex $L_{R/\FF_p}$
		is concentrated in degree zero.  So  \cite[Corollary 5.5]{Fink2025}  (or, more precisely, its proof) implies that $f$ is formally \'etale.
	\end{proof}
	
	%\medskip
	
	\begin{Remark} The above proof shows that for any $x\in X$ a $p$-basis of  $\cO_{X,x}$ can be constructed by taking a minimal set of generators of the maximal ideal $m_x$ and adding to it lifts of a $p$-basis of $\kappa(x)$. A weaker form of this statement is proven as \cite[Theorem 3.1]{Kimura-Niitsuma1980}, where the authors choose a special minimal set of generators of $m_x$. \cite[Theorem 1]{Tyc1988} shows that these two facts are equivalent.
	\end{Remark}
	
	%\medskip
	\begin{Example}
		Here we show that the map $\hat{\cO}_{X,x}\to \hat {\cO}_{\AA^n_K, f(x)}$, induced by $f$ on the completions of local rings, need not be an isomorphism.
		Let $k=\FF_p ((x))$ be the field of formal Laurent series. This field is $F$-finite and $\{x\}$ is its $p$-basis. Then
		for $X=\Spec k$ the above morphism $X\to \AA ^1_{\FF_p}$ corresponds to the inclusion $\FF_p[x]\hookrightarrow \FF_p ((x))$. So the only point of $X$ is mapped to the generic point of $\AA ^1_{\FF_p}$ and the corresponding map on local rings is given by the inclusion  $\FF_p(x)\hookrightarrow \FF_p ((x))$, so it is not an isomorphism on the completions of local rings.
	\end{Example}

%	\medskip

Let $\cD_X$  denote the ring of differential operators on $X$ and let $\cD_X^{(s)}$ be the centralizer of $\cO_X ^{p^s}$ in $\cD_X$. The following lemma can be proven in the same way as  \cite[Lemma 3.3]{Chase1974}.

\begin{Proposition}
	If $X$ is a Noetherian regular $F$-finite $\FF_p$-scheme then the action of $\cD_X$ on $\cO_X$ induces an isomorphism 
	$\cD_X^{(s)}\to \cEnd _{\cO_X^{p^s}}(\cO_X)$ and $\cD_X=\bigcup _{s\ge 0} \cD_X^{(s)}$.
\end{Proposition}

Note that if $R$ is a Noetherian regular $F$-finite $\FF_p$-algebra then equality $D_R=\bigcup _{s\ge 0}\End
_{R^{p^s}}R$, which implies that $R^{\sf perf}$ is the center of $D_R$.

%\medskip

Let   $\cD _X{\sf \text{-} Coh}$ be the category of left $\cD_X$-modules, which are coherent as $\cO_X$-modules. Using the above proposition one can explicitly write down Morita's equivalence of $\cD_X^{(s)}$ and $\cO_X^{p^s}$ as in \cite[Proposition 2.1]{Alvarez-Montaner-Blickle-Lyubeznik2005}. 
This can be used to generalize Katz's theorem \cite[Theorem 1.3]{Gieseker1975} to the following result:

\begin{Theorem}
	Let $X$ be a Noetherian regular $F$-finite $\FF_p$-scheme. Then there is an equivalence of categories between ${\sf Vect} ^{\sf perf}(X )$ and $\cD _X{\sf \text{-} Coh}$.
\end{Theorem}

\begin{Remark} \label{Kindler-works-nonetheless}
    Using the above theorem and Proposition \ref{existence-of-local-p-coordinates} one can use $D$-modules to
give another proof of Lemma \ref{full-faithfullness-of-restriction-to-open-regular} in the spirit of proof of 
 \cite[Lemma 2.5]{Kindler2015}.	
\end{Remark}

%\medskip

In the following we will not use the above theorem but instead we write down a general version of Cartier's descent. This version essentially follows from \cite[Proposition 2.1]{Alvarez-Montaner-Blickle-Lyubeznik2005} (see \cite[footnote on p.~462]{Alvarez-Montaner-Blickle-Lyubeznik2005}, which however seems to require some additional work) but we follow the standard proof contained in \cite[Theorem 5.1]{Katz1970}.

\begin{Theorem}\label{Cartier's-descent} 
	Let $X$ be a Noetherian regular $F$-finite $\FF_p$-scheme. Then the functor $F_X^*: \Coh (X)\to \Mic^{0}(X)$ given by sending $E$ to $F_X^*E$ with the canonical connection  is an equivalence of categories between the category of coherent $\cO_X$-modules and the category of coherent $\cO_X$-modules with an integrable connection and zero $p$-curvature. Analogous fact holds also for quasi-coherent $\cO_X$-modules.
\end{Theorem}

\begin{proof}
	The quasi-inverse $ \Mic^{0}(X)\to \Coh (X)$ to $F_X^*$ is given by  sending $(E, \nabla)$ to the sheaf of horizontal sections $E^{\nabla}$ treated as an $\cO_X$-module by inducing an $\cO_X^p$-module from $E$. 
	We need to show that for any object  $(E,\nabla) $ of $\Mic^{0}(X)$ the canonical map of $\cO_X$-modules
	$F^*_X(E^{\nabla})\to E$ 
	is an isomorphism. Since the question is local we can reduce the problem to proving an analogous isomorphism for modules over a Noetherian regular local $F$-finite $\FF_p$-ring $R$. We can choose a  $p$-basis $\{r_1,...,r_n\}$ of $R$ over $R^p$ so that $\Omega_R$ is a free $R$-module with basis $\{ dr_1,...,dr_n\}$.
	Then the proof continues as the proof of \cite[Theorem 5.1]{Katz1970}. 
\end{proof}

\section{ Local behavior of the $F$-divided fundamental gerbe revisited}

In this section, we revisit the local behavior of the $F$-divided fundamental
gerbe discussed in \S \ref{Local behavior
regular}. We will see that the ``extension of subobjects'' property holds in considerable
generality, while ``full faithfulness'' requires certain
normality conditions. It is well-known that for a normal variety  $X$, there exists a natural surjection $\Gal(K(X))\to\pi_1^\et(X)$ (see \cite[Expos\'e V, Proposition 8.2]{SGA1}). The above two properties yield an analogous result for the $F$-divided fundamental gerbe. 

 \medskip
	
The following lemma generalizes Lemma \ref{extending objects in the regular case} to arbitrary schemes.
	
	\begin{Lemma}\label{extending objects in general}
		Let $X$ be a connected Noetherian $F$-finite $\FF_p$-scheme and
		let $\imath : U\hookrightarrow X$ be a dense open subset. Then for any
		$\EE\in\Vect^\perf(X)$ and a subobject $\GG_U\subseteq \imath^*\EE$ there exists a unique subobject $\GG\subseteq \EE$ such that	$\imath^*\GG=\GG_U$ as subobjects of $\imath^*\EE$. 
	\end{Lemma}
	
	\begin{proof} 
		Let us recall that $X$ is excellent. By Gabber's theorem (see
        \cite[1.1]{Gabber2005} and \cite[II, Th\'eor\`eme
        4.3.1]{Illusie-Laszlo-Orgogozo2014}), which generalizes de Jong’s
        result \cite[Theorem 4.1]{deJong1996} on the existence of alterations
        from smooth varieties, there exists a covering $f: \tilde X\to X$ of
        $X$ in $h$-topology such that $\tilde X$ is a finite disjoint union of
        integral regular schemes (which are automatically also Noetherian and
        $F$-finite).
        By  Corollary \ref{Kindler} and Lemma \ref{relative gerbe =
        surjective}, we can extend $(f|_{f^{-1}(U)})^*\GG_U\subset
        (f^*\EE)|_{f^{-1}(U)}$ to an $F$-divided subbundle $\tilde\GG\subset f^*\EE$. Now we can use Theorem \ref{h-descent-for-F-divided} to show that
		$\tilde\GG$ descends to  an $F$-divided subbundle of $\EE$. More precisely, $ f^*\EE$ comes with the canonical descent datum for $f$ given by the canonical isomorphism $\alpha: {\sf pr}_1^*f^*\EE\stackrel{\simeq}{\longrightarrow} {\sf pr}_2^*f^*\EE$ on $\tilde X\times_ X\tilde X$. 
		Note that by construction the composition 
		$${\sf pr}_1^*\tilde\GG\subset {\sf pr}_1^* f^*\EE\mathop{\longrightarrow}^{\alpha} {\sf pr}_2^*f^*\EE\to  {\sf pr}_2^*
		(f^*\EE/\tilde\GG)$$
		vanishes on $f^{-1}(U)\times _U f^{-1}(U)$. So it is also vanishes on $\tilde X\times_ X\tilde X$ and hence it defines a map ${\sf pr}_1^*\tilde\GG\mathop{\to} {\sf pr}_2^*\tilde\GG$. This map is an isomorphism, as it is an isomorphism after restricting to $f^{-1}(U)\times _U f^{-1}(U)$. So it defines a descent datum for the inclusion $\tilde\GG\subset f^*\EE$ which by   	Theorem \ref{h-descent-for-F-divided} gives the required $F$-divided subbundle of $\EE$.
		Uniqueness of $\GG$ follows from Theorem
        \ref{F-divided-sheaves-form-Tannakian-category}, (B).
	\end{proof}

    {
        The above Lemma easily generalizes to algebraic stacks:
\begin{Lemma}\label{extending objects in general stack version}
		Let $X$ be a connected Noetherian weakly $F$-finite $\FF_p$-algebraic
        stack and
		let $\imath : U\hookrightarrow X$ be a dense open { substack.}
        Then for any
		$\EE\in\Vect^\perf(X)$ and a subobject $\GG_U\subseteq \imath^*\EE$ there exists a unique subobject $\GG\subseteq \EE$ such that	$\imath^*\GG=\GG_U$ as subobjects of $\imath^*\EE$. 
	\end{Lemma}
    \begin{proof} Let $f\colon \tilde{X}\to X$ be a smooth atlas, where
        $\tilde{X}$ is a Noetherian $F$-finite scheme. By Lemma \ref{extending
        objects in general}, the subobject $f^*\GG_U\subseteq
        f^*\imath^*\EE$ extends to a subobject $\tilde\GG\subseteq f^*\EE$.
        Running the same argument of Lemma \ref{extending objects in general}, we see that $\tilde\GG$ descends to a subobject
        $\GG\subseteq\EE$ extending $\GG_U$.
    \end{proof}
    }

\begin{Theorem}\label{normal open induces surjection on stratified fundamental groups}
		Let $X$ be an irreducible Noetherian geometrically unibranch $F$-finite
        $\FF_p$-scheme and let $\eta $ be the
        generic point of $X$. Then the induced $1$-morphism $\Pi_\eta^{\Fdiv}\to
        \Pi_X^{\Fdiv}$ is a relative gerbe over the field $\cO_X(X)^\perf$. In
        particular, for any dense open $U\subseteq X$, $\Pi_U^{\Fdiv}\to
        \Pi_X^{\Fdiv}$ is a relative gerbe over the field $\cO_X(X)^\perf$.
\end{Theorem}

\begin{proof}
        By \cite[Lemma \href{http://stacks.math.columbia.edu/tag/035R}{035R}
        and Lemma \href{http://stacks.math.columbia.edu/tag/035R}{0GIQ}]{StacksProject}  the normalization morphism $X^{\nu}\to X$ is a finite universal homeomorphism. 
So by Lemma \ref{passing-to-universal-homeomorphism} we can assume that $X$ is normal. 

Consider $U=X_{\sf reg}$. Note that this set is open in $X$ as $X$ is excellent. 	
Since the complement of $U$ in $X$ has codimension $\ge 2$ and $X$ is normal, we have
$\imath_*\cO_U=\cO_X$. Then the projection formula implies that  the functor
${\sf Vect} ^{\sf perf}(X)\to {\sf Vect} ^{\sf perf}(U)$ is fully faithful. In
this case,  $\Pi_U^{\Fdiv}\to \Pi_X^{\Fdiv}$ is a relative gerbe by
Lemma \ref{extending objects in general} and Lemma \ref{relative gerbe = surjective}.
		
To finish the proof, we consider the following composition:
\[\PiF_{\eta}\longrightarrow\PiF
_{X_{\sf reg}}\longrightarrow\PiF_X\]
where the left arrow is a relative gerbe by Corollary \ref{Kindler}; the
        right arrow is a relative gerbe by the above, so the composition is
        also a relative gerbe.  
\end{proof}

Theorem \ref{normal open induces surjection on stratified fundamental groups}
also generalizes to algebraic stacks. Recall that an algebraic stack $X$ is called
\emph{geometrically unibranch} if it admits a smooth atlas $U\twoheadrightarrow
X$, where $U$ is a geometrically unibranch scheme. { This is well-defined thanks
to
\cite[\href{https://stacks.math.columbia.edu/tag/0DQ2}{0DQ2}]{StacksProject}.}

\begin{Theorem} \label{normal open induces surjection on stratified fundamental
    groups stack case} Let $X$ be an irreducible Noetherian geometrically unibranch
    weakly $F$-finite
        $\FF_p$-algebraic stack and let $X'\subset X$ a dense open substack, then
        $\Pi_{X'}^{\Fdiv}\to
        \Pi_X^{\Fdiv}$ is a relative gerbe over the field $\cO_X(X)^\perf$.
\end{Theorem}
\begin{proof} Let us choose a smooth groupoid presentation $R\rightrightarrows U$ of $X$
    by algebraic spaces, where $U$ is a Noetherian geometrically
    unibranch $F$-finite scheme. Restricting the presentation to $X'$ we
    get a presentation $R'\rightrightarrows U'$ of $X'$. Note that by
    construction, $U'\subset U$ and $R'\subset R$ are dense opens.
Moreover,   $U, R$ are $F$-finite Noetherian geometrically unibranch algebraic spaces,
    so by Theorem \ref{F-divided-sheaves-form-Tannakian-category} and Theorem
    \ref{normal open induces surjection on stratified fundamental groups}, the
    restriction functor $\Vect^\perf(R)\to\Vect^\perf(R')$ is faithful, while
    $\Vect^\perf(U)\to\Vect^\perf(U')$ is fully faithful. Applying fpqc
    descent of $\Vect^\perf(-)$ to the presentations $R\rightrightarrows U$ and
    $R'\rightrightarrows U'$ one sees that the restriction functor
    $\Vect^\perf(X)\to\Vect^\perf(X')$ is fully faithful. This, together with Lemma
    \ref{extending objects in general stack version}, completes the proof.
\end{proof}

Let $X, Y$ be irreducible schemes. Let us recall that a morphism $f: X\to Y$ is called \emph{birational}
if it induces an isomorphism of the function fields (see \cite[Definition \href{https://stacks.math.columbia.edu/tag/01RO}{01RO}]{StacksProject} for a more general definition).

\begin{Corollary}\label{birational-map-on-gerbes}
    Let $f\colon \tilde X\to X$ be a birational morphism of finite type between irreducible $F$-finite
Noetherian $\FF_p$-schemes with $X$ being geometrically unibranch. Then the induced morphism $f_*: \PiF_{\tilde X}\to\PiF_{X} $ is a relative gerbe over $\cO_X(X)^{\sf perf}$. 
\end{Corollary}

\begin{proof}
    By \cite[Lemma
    \href{https://stacks.math.columbia.edu/tag/0BAC}{0BAC}]{StacksProject},
    there exists a non-empty open subset $U\subset X$ such that
    $f|_{f^{-1}(U)}$ is an isomorphism. Note that both $U$ and $f^{-1}(U)$ are
    irreducible, so by  Theorem \ref{F-divided-sheaves-form-Tannakian-category},
    their $F$-divided gerbes exist.
Therefore we have induced maps 
$\PiF_{f^{-1}(U)}\to\PiF_{\tilde X}$ and  $\PiF_{f^{-1}(U)}\to\PiF_{X}$, where
the latter is a relative gerbe
by Theorem \ref{normal open induces surjection on stratified fundamental
groups}. So $f_*: \PiF_{\tilde X}\to\PiF_{X} $ is also a relative gerbe. 
\end{proof}

	\begin{Corollary}\label{subobject-of-F-divided-bundle}
        { Let $X$ be an irreducible Noetherian normal  weakly $F$-finite $\FF_p$-algebraic
        stack.} Then 
${\sf Vect} ^{\sf perf}(X)$ is a Serre subcategory of ${\sf Ref} ^{\sf perf}(X)$.
	\end{Corollary}
	\begin{proof} 
Clearly, the subcategory 	${\sf Vect} ^{\sf perf}(X)$ is closed under extensions, so we only need to show that it is closed under taking subobjects and quotients. Since giving a quotient object is equivalent to giving a subobject of the dual, we can restrict to considering subobjects. 
Let $i\colon X_{\sf reg}\subseteq X$ be the regular locus of $X$.
By Lemma \ref{description-of-Ref-perf} we need to show that for any object $\EE$ of $ {\sf Vect} ^{\sf perf}(X)$ and a subobject $\FF _U\subset i^*\EE$ there exists a unique subobject $\FF \subset \EE$ such that
$i^*\FF= \FF_U$ as subobjects of $i^*\EE$. This follows from { Lemma
\ref{extending objects in general stack version}}. \end{proof}

\begin{Remark}
For general  irreducible Noetherian geometrically unibranch $F$-finite $\FF_p$-scheme $X$, the above proof gives only the fact that ${\sf Vect} ^{\sf perf}(X)$ is a Serre subcategory of ${\sf Vect} ^{\sf perf}((X_{\sf red})_{\sf reg})$.
\end{Remark}

%\medskip

Let $\EE$ and $\FF$ be $F$-divided vector bundles on a Noetherian $\FF_p$-scheme $X$.
Let us recall that one can define the sheaf $\cHom (\FF , \EE)$ as the $F$-divided bundle
with  $\cHom (\FF , \EE)_i= \cHom (F_i , E_i)$ and obvious induced isomorphisms. 
We also  define \emph{the sheaf $\cHom ^h(\FF , \EE)$ of horizontal maps from $\FF$ to $\EE$} by
setting
$$\left( \cHom ^h(\FF, \EE)\right) (U):= \Hom _U(\FF|_U , \EE| _U)$$
for any open $U\subset X$ (note that this presheaf is a sheaf).

We also define \emph{the sheaf $\EE^h$ of horizontal sections of $\EE$} as the sheaf $\cHom ^h(\unit_X , \EE)$. 
If we write $\EE= \{E_i, \sigma_i\}$ then we have
\[\EE^h= \lim \left(... \stackrel{\tau_2}{\longrightarrow}E_2
\stackrel{\tau_1}{\longrightarrow}E_1
\stackrel{\tau_0}{\longrightarrow}E_0\right),\]
where $\tau_i$ is the composition of the natural map $E_{i+1}\to F_X^*E_{i+1}$ with $\sigma_i$.
In particular, our definition agrees with the one from \cite[\S1]{Gieseker1975}.
Note that both $\cHom ^h(\FF, \EE)$ and $\EE^h$ are sheaves of modules over the sheaf of rings $\unit _X^h=\cO_X^{\sf perf}.$

    \begin{Lemma}\label{evaluation map - general nonsense}
Let  $\Gamma$ be an affine gerbe over $k$, and let
$\alpha\colon\Gamma\to\Spec(k)$ be the projection map. Let $E\in\Vect(\Gamma)$
be a vector bundle. Then $\Hl^0(E)$ is a finite dimensional $k$-vector space,
and the adjunction map
\begin{equation}\label{adjunction}
   \alpha^*\alpha_*E\longrightarrow E
\end{equation}
is injective.
\end{Lemma}
\begin{proof} The equation \eqref{adjunction} is injective iff there exists a
    field
    extension $l/k$ such that \eqref{adjunction}$\otimes_kl$ is injective, {
    and we have $\Hl^0(E\otimes_kl)=\Hl^0(E)\otimes_kl$}. Thus
    we may assume that $\Gamma=\cB_kG$ for some affine group scheme $G$. In
    this case $E$ is a finite dimensional $G$-representation, and
    \eqref{adjunction} is nothing but the inclusion
    $E^G\subseteq E$, which is injective. For the first claim, it is enough to observe
    that {$\Hl^0(E)=\alpha_*E=E^G$}.
\end{proof}

\begin{Lemma}\label{evaulation-map}
Let  $X$ be a connected Noetherian  $F$-finite $\FF_p$-scheme and let $k=\cO_X(X)^{\sf perf}$. Then 
for any $\EE\in \Vect^{\perf}(X)$,
$\EE ^h(X)$
is a finite dimensional $k$-vector space and the canonical evaluation map 	
$$\EE^h(X)\otimes _k \unit_X\to \EE$$ 
is injective.
\end{Lemma}
\begin{proof} In Lemma \ref{evaluation map - general nonsense}, we take
    $\Gamma\coloneqq \Pi_X^\Fdiv$. Then viewing $\EE$ as a vector bundle on
    $\Gamma$, we have $\EE^h(X)=\Hl^0(\EE)$ and the
    evaluation map is the adjunction map \eqref{adjunction}.
\end{proof}

\begin{Corollary}
Let  $X$ be an irreducible Noetherian  geometrically unibranch $F$-finite $\FF_p$-scheme.
Then $\unit _X^h$ is the constant sheaf associated to the field $k= \cO_X(X)^{\sf perf}$, $\cHom ^h(\FF, \EE)$
is the constant sheaf associated to the finite dimensional $k$-vector space $\Hom _X(\FF , \EE)$,
and $\EE^h$ is the constant sheaf associated to the finite dimensional $k$-vector space $\EE^h (X)$.
\end{Corollary}

\begin{proof}
By Theorem  \ref{F-divided-sheaves-form-Tannakian-category} (C), we know that $k=\End \unit_X$ and it is a field.
Finite dimensionality of  $\Hom _X(\FF , \EE)$ follows from the previous lemma and the fact that 
	$\Hom _X(\FF , \EE)=  \left(\cHom (\FF, \EE)\right)^h (X)$. 
Now the corollary follows from the fact that by Theorem \ref{normal open induces surjection on stratified fundamental groups} for any open subset $U\subset X$ the restriction functor ${\sf Vect} ^{\sf perf}(X)\to {\sf Vect} ^{\sf perf}(U)$ is fully faithful.
\end{proof}

\section{$F$-divided vector bundles on projective varieties}

{ This section revisits an important technique from \cite{Esnault-Mehta2010}, which allows one to relate the constituent bundles of an $F$-divided sheaf to points in a moduli space. We generalize this technique from smooth projective varieties to the setting of normal projective varieties.}

{ Let us fix a perfect field $k$ of positive characteristic.}
In the following we use the notation from \cite{Fulton1998}. Let $X$ be  a proper connected $k$-scheme and let $A_i(X)$ be the Chow group of dimension $i$ cycles on $X$. We have a well-defined degree map $\int_X: 
A_{*}(X)\to \ZZ$. For a vector bundle $E$ on $X$ one defines operational Chern classes $c_i(E)\cap (\cdot): A_*(X)\to A_{*-i}(X)$. Then any polynomial $P(E)$ in the Chern classes of $E$ operates on $A_*(X)$.
We say that the Chern classes of $E$ \emph{vanish numerically} if for any class $\alpha\in A_*(X)$ and any homogeneous polynomial $P(E)$ of degree $>0$ in the Chern classes of $E$ we have $\int _XP(E)\cap \alpha=0$.
The proof of the following lemma is the same as that of \cite[Lemma 2.1]{Esnault-Mehta2010}.

\begin{Lemma}\label{vanishing-Chern-classes}
Let $X$ be a connected proper $k$-scheme. Then for any $\EE= \{E_n, \sigma_n\}\in \Vect^{\perf}(X)$
the Chern classes of $E_n$ vanish numerically. 
\end{Lemma}	

This lemma together with \cite[Corollary 18.3.1 and Example 3.2.3]{Fulton1998} shows the following corollary:

\begin{Corollary}
	Let $X$ be a connected proper $k$-scheme  and let  $\EE= \{E_n, \sigma_n\}\in \Vect^{\perf}(X)$ be an $F$-divided bundle of rank $r$. Then for any vector bundle $E'$ on $X$ and any $n\ge 0$
	we have	$$\chi (X, E_n\otimes E')= r \chi (X, E').$$ In particular, if $X$ is projective then for any ample line bundle the normalized Hilbert polynomial of $E_n$ is equal to the Hilbert polynomial of $\cO_X$.
\end{Corollary}

The following proposition generalizes \cite[Proposition 2.3]{Esnault-Mehta2010} from smooth to normal varieties. Note that the proof in  \cite[Proposition 2.3]{Esnault-Mehta2010} is not completely correct and one can find a rather complicated correction in \cite{Esnault-Mehta-erratum}. We give a different much simpler argument.
	
\begin{Proposition}\label{Jordan-Holder-for-F-divided-bundles}
	Let $X$ be a connected normal projective $k$-variety and let us fix an ample line bundle. Then for any $\EE\in \Vect^{\perf}(X)$ there exists $n_0\in \ZZ_{\ge 0}$ such that $\EE$ is a successive extension of $F$-divided bundles $\FF\in \Vect^{\perf}(X) $  with the property that all $F_n$ with $n\ge n_0$ are slope stable with numerically vanishing Chern classes. 
\end{Proposition}
	
\begin{proof}
The beginning of the proof is the same as in \cite{Esnault-Mehta2010}.
Namely, we can find some $N$ such that for all $n\ge N$ the bundles $E_n$ are slope semistable.
Without loss of generality we can assume that $N=0$.
Let $E_{n0}=0\subset E_{n1}\subset ...\subset E_{nl_n}=E_n$ be a  Jordan-H\"older filtration of $E_n$.
By definition all quotients $E_n^i:= E_{ni}/E_{n(i-1)}$ are torsion free and slope stable of slope $0$  
and the isomorphism class of the reflexivization of the associated graded $\bigoplus_{i=1}^{l_n} E_{n}^i$ does not depend on the choice of a Jordan-H\"older filtration of $E_n$.
In particular, the length $l_n$ is a well-defined number.
Since $E_{n}\simeq F^*_X{E_{n+1}}$ is slope semistable, we can obtain a Jordan-H\"older filtration of $E_n$
as a refinement of the reflexive pullback $F_X^{[*]}E_{(n+1)\bullet}$ of a Jordan-H\"older filtration of $E_{n+1}$. In particular, we have $l_n\ge l_{n+1}$. Let $l=\min_{n\in \ZZ_{\ge 0}} l_n$. Again replacing $\EE$
by its shift we can assume that $l_n=l$ for all $n\in \ZZ_{\ge 0}$. This implies that for any $n\ge m\ge 1$ and $j$, the reflexive pullback $F_X^{[m]}(E_n^j)^{**}$ is slope stable. Now we can proceed as in \cite{Esnault-Mehta2010}.
Namely, let $S_n\subset E_n$ be the socle of $E_n$, which is the maximal nontrivial subsheaf which is slope polystable of slope $0$ (note that it is reflexive as $E_n/S_n$ has to be torsion free). 
Then $F_X^{[n]}S_n$ is again slope polystable of slope $0$ and hence it is contained in $S_0$.
Therefore we get a decreasing sequence 
$$\cdots\subset F_X^{[n+1]} S_{n+1}\subset F_X^{[n]}S_n\subset \cdots \subset E_0,$$
which becomes  stationary for large $n$ as the inclusions are either equalities or the rank drops.
This shows that there exists some $N$ such that  $F_X^{[*]}S_{n+1}\simeq S_n$ for all $n\ge N$.
Then we define $\SS '\subset \EE$ by setting  
\[S_n'=\left\{\begin{array}{cl}
F_X^{[N-n]}S_N &\hbox{ for }0\le n\le N,\\
S_n& \hbox{ for } n> N.
\end{array}
\right.\] 
By Corollary \ref{subobject-of-F-divided-bundle} we know that $\SS'$ is an $F$-divided subbundle of $\EE$ and
for all $n\ge N$ the bundles $S'_n$ are slope polystable of slope $0$. As at the end of proof of \cite[Proposition 2.3]{Esnault-Mehta2010} it is easy to see that $\SS'$ is a direct sum of  $F$-divided bundles $\SS ^j$ for which $S^j_n$ are slope stable of slope $0$ for large $n$.
In this case all Chern classes of the factors vanish numerically by Lemma \ref{vanishing-Chern-classes} and we can finish the proof by induction on the rank of $\EE$.
\end{proof}

	\section{Gieseker's conjecture for normal proper varieties}
    { In this section, we generalize the main theorem of \cite{Esnault-Mehta2010} from smooth projective varieties to normal proper varieties, leveraging the tools developed in the previous sections.}
	
	\medskip
	
	The following lemma is a relative version of flattenning stratification, and it was was proven in \cite[Expos\'e XII, Lemme 4.4]{SGA6}.
	
	\begin{Lemma}\label{relative-flattening}
		Let $X\to S$ and $Y\to S$ be proper morphisms of finite presentation, and let $f: X\to Y$ be a morphism. Let $E$ be a sheaf on $X$ of finite presentation and flat over $S$.
		Then there exists a surjective monomorphism $\tilde S\to S$ of finite presentation such that 	given a Cartesian diagram
		$$
		\xymatrix{
			X_T\ar@/^1pc/[rr]\ar[r]^{f_T}\ar[d]\ar[r]&Y_T\ar[r]\ar[d]&T\ar[d]\\
			X\ar[r]^{f}\ar@/_1pc/[rr]&Y\ar[r]&S\\
		}
		$$
		$R^i(f_T)_* E_T$ is flat over $T$  for all $i\ge 0$ if and only if $T\to S$ factors through $\tilde S$.	
	\end{Lemma}

	We also have the following result, which is proven in \cite[Expos\'e XII, Lemme 4.6]{SGA6}.	
	
	\begin{Lemma}\label{relative-isomorphism-of-sheaves}
		Let $X\to S$ be a proper morphism of finite presentation and
		let $u: E_1\to E_2$ be a homomorphism of $\cO_Y$-modules of finite presentation with $E_2$ flat over $S$. Then 
		there exists an open subscheme $\tilde S\subset S$ of finite presentation such that given a Cartesian diagram
		$$
		\xymatrix{
			X_T\ar[r]^{\bar{h}}\ar[d]&X\ar[d]\\
			T\ar[r]^{h}&S\\
		}
		$$
		$\bar{h}^*u$ is an isomorphism if and only if $h$ factors through $\tilde S$. %:  (\bar{h})^*E_1\to  (\bar{h})^*E_2
	\end{Lemma}
	
	The proof of the next theorem relies on the existence of a certain moduli scheme of vector bundles. Let us  fix a positive integer $r$, a projective morphism $f: X_S\to S$ of schemes of finite type over some fixed field (that in our case will be $\FF_p$) and an $f$-very ample line bundle $\cO_{X}(1)$. Then by \cite[Theorem 4.1]{Langer2004}
	there exists a quasi-projective moduli scheme $M(r, X_S)\to S$ of Gieseker stable rank $r$ vector bundles  with numerically vanishing Chern classes. This moduli scheme universally corepresents the functor of isomorphism classes of flat families of geometrically Gieseker stable rank $r$ vector bundles (with numerically vanishing Chern classes) on the fibers of $f$. In the following we  use also existence of a quasi-universal family $\cU_S$. This is a flat vector bundle on $M(r, X_S)\times _S X_S\to M(r, X_S)$, which is geometrically Gieseker stable on the fibers and such that for any $T\to S$ and a $T$-flat family $E_T$ of Gieseker stable rank $r$ vector bundles with numerically vanishing Chern classes on the fibers of $f_T: X_T\to T$, if $\varphi _{E_T}: T\to M(r, X_S)$  denotes the classifying morphism then there exists a vector bundle $W$ on $T$ such that $E_T\otimes f_T^*W\simeq \varphi _{E_T}^*\cU_S$.

    \begin{Theorem}\label{EM10 for normal proper}
		Let $X$ be a normal integral $\FF_p$-scheme which is proper geometrically connected over some perfect field $k$.
		If for some rational point  $x\in X(k)$ the maximal \'etale quotient  	\(\pi_1^{\No,\et}(X,x)\) of the Nori fundamental group scheme vanishes 	then there are no non-trivial $F$-divided bundles on $X$.	
	\end{Theorem}
	
	\begin{proof}				
        By Chow's lemma there exists a normal integral projective $k$-scheme and a surjective birational $k$-morphism $f: \tilde X\to X$.  
 By Lemma \ref{evaulation-map},  if an $F$-divided bundle $\EE_{\bar k}$ is trivial then $\EE$ is trivial. Moreover, we have canonical isomorphisms	\(\pi_1^{\et}(X_{\bar k},\bar{x})\simeq \pi_1^{\No,\et}(X_{\bar k},\bar{x})\simeq \pi_1^{\No,\et}(X_{\bar k},\bar{x})_{\bar k}\).
		So in the following we can assume that $k$ is algebraically closed and $\pi_1^{\et}(X_{\bar k},\bar{x})=0$.  By \cite[Proposition 2.4]{Esnault-Mehta2010} (note that the proof of this proposition works for any geometrically connected  proper $k$-scheme) it is sufficient to prove that every simple object in $\Vect ^{\perf} (X)$ is trivial.
		
		It is easy to see that every rank one  $F$-divided bundle $\EE$ on $X$ is trivial so we need to consider
		a rank $r\ge 2$ $F$-divided bundle $\EE$ on $X$, which is a simple object in $\Vect ^{\perf} (X)$. Then by Lemma \ref{birational-map-on-gerbes} $f^*\EE$ is also a simple object in $\Vect ^{\perf} (\tilde X)$. In the following we fix some ample line bundle $\cO_{\tilde X}(1)$.
		
		By Proposition \ref{Jordan-Holder-for-F-divided-bundles} there exists some $n_0$ such that all $f^*E_n$ with $n\ge n_0$ are slope stable (and hence Gieseker stable) with numerically vanishing Chern classes. Without loss of generality we can assume that $n_0=0$. Let us consider the moduli scheme $M (r, \tilde X)/k$ of Gieseker stable rank $r=\rk \EE$  vector bundles on $\tilde X/k$ with numerically vanishing Chern classes.	
		
		Let us define the locus $A_j\subset M(r, \tilde X)$, which is the Zariski closure of the set $\{ [f^*E_n]\} _{n\ge j}\subset M(r, \tilde X) (k)$. Since $A_{j+1}\subset A_j$, the sequence $\{ A_j\}_{j\ge 0}$ stabilizes and for large $n$ we have $A_n=A:=\bigcap A_j$. Let $M_0$ be the open subset of $ M(r, \tilde X)$ corresponding to bundles $E$ such that $F_{\tilde X}^*E$ is Gieseker stable. Then pullback by the absolute Frobenius morphism defines a morphism $M_0\to  M(r, \tilde X)$. After restricting to $A\cap M_0$ we have a well defined morphism, which gives a dominant
		rational map $\psi: A\dashrightarrow A$. This morphism is not $k$-linear and to make it $k$-linear we need to consider the relative Frobenius morphism $F_{\tilde X/k}: \tilde X\to \tilde X'$. Then we have 
		$ M(r, \tilde X')\stackrel{\simeq}{\longrightarrow}  M(r, \tilde X)\times _{F_k}k$ and if  $A'$ is the reduced preimage of $A\times _{F_k}k$ then we obtain a dominant rational $k$-morphism $\varphi: A'\dashrightarrow A$, which is the restriction of the classical Verschiebung rational map.
		
		We need to spread out the whole situation. There exists a finitely generated $\FF_p$-algebra $R\subset k$ and a  scheme $X_S$ of finite type over $S=\Spec R$ such that $X/k$ is isomorphic to the generic geometric fiber of $X_S\to S$.  Shrinking $S$ if necessary we can assume that $X_S\to S$ is proper and flat with geometrically connected fibres. We can also assume that all fibers of $X_S\to S$ are geometrically normal and geometrically integral. 
		Similarly, we can find $f_S: \tilde X_S\to X_S$ with $\tilde X$ projective  over $S$ so that $f_S$ is isomorphic to $f$ over the generic geometric point of $S$. By Zariski's main theorem for all $s\in S$ the morphisms  $f_s: \tilde X_s\to X_s$ are birational and satisfy $f_{s*}\cO _{\tilde X_s}=\cO _{X_s}$.  We can also assume that $\tilde X_S\to S$ is flat and all its fibers  are geometrically normal and geometrically integral.

		We can also construct $S$-flat models $A_S\subset M(r, \tilde X_S)$ for $A$ and $A_S'\subset M(r, \tilde X'_S)$ for $A'$.
		We have a dominant rational map of $S$-schemes $\varphi_S:A'_S\dashrightarrow A_S$ extending $\varphi$ and
		defined by pullback via the relative Frobenius morphism  $F_{\tilde X_S/S}:\tilde X_S\to \tilde X'_S$.
		Shrinking $S$ if necessary we can assume that the restriction $\varphi_s: A'_s\dashrightarrow A_s$ is a dominant rational map for all closed points $s$ of $S$. For such $s$ we let  $\varphi_s^{(i)}$ denote the $i$th Frobenius twist of $\varphi_s$ and set $m_s=(\kappa (s):\FF_p)$. Since $\varphi _s$ is defined by $[E]\to [F_{\tilde X_s/s}^*E]$, the composition $\varphi_s^{(m_s-1)}\circ ...\circ \varphi_s^{(1)}\circ \varphi _s$ defines a rational endomorphism $A_s\dashrightarrow A_s$ for any closed point $s\in S$. Now as in \cite{Esnault-Mehta2010} Hrushovski's theorem \cite[Corollary 1.2]{Hrushovski2004} implies that set of closed points of $A_s$, which are periodic for this rational endomorphism is dense in $A_s$. Such points correspond to geometrically Gieseker stable vector bundles $G$ on $\tilde X_s$  such that for some $m\ge 1$ we have an isomorphism $(F^m_{\tilde X_s})^*G\simeq G$.

		Let us recall that the moduli scheme $M(r, \tilde X_S)/S$ comes equipped with  a quasi-universal family $\cU$. Applying 	Lemma \ref{relative-flattening} to the family $\cU_{A_S}:=\cU|_{A_S\times _S{\tilde X}_S}$ and the morphism 
		$f_{A_S}$ over $A_S$, we obtain a surjective monomorphism $\mu: \tilde
        A_S\to A_S$ of finite presentation satisfying the conclusion of the
        lemma. By openness of the flat locus (see \cite[Theorem
        \href{https://stacks.math.columbia.edu/tag/0399}{0399}]{StacksProject}) the set 
        $U:=\{ x\in \tilde A_S : \mu \hbox{ is flat at }x \}$ is open and dense
        in $\tilde A_S$. On the other hand, by \cite[Theorem
        \href{https://stacks.math.columbia.edu/tag/025G}{025G}]{StacksProject}
        $U\to A_S$ is an open immersion so the set $U\cap \{ [f^*E_n]\}$ is
        dense in $A$.
        Since $f_*(\cU_{A_S}|_{ [f^*E_n]\times \tilde X})=E_n$,  this implies
        that $f_*(\cU_{A_S}|_{(U\cap A)\times \tilde X})$ is locally free. So shrinking $S$ if necessary we can assume that  $f_{U*}\cU_U$, where $\cU_U=\cU_{A_S}|_{U\times _S\tilde{X}_S}$, is locally free. Now applying Lemma \ref{relative-isomorphism-of-sheaves} to $U\times _S\tilde{X}_S\to U$ and the relative evaluation homomorphism $u: f_U^*f_{U*}\cU_{U}\to \cU_{U}$, we obtain an open subscheme $V\subset U$ such  that   $\overline{h}^*u$ is an isomorphism for any $T\to V$. But $V$ contains images of the points corresponding to $[f^*E_n]$, so $V\cap A_s$ is open and dense in $A_s$ for all $s\in S$.
		This shows that there are  geometrically Gieseker stable vector bundles $G$ on $\tilde X_s$  such that for some $m\ge 1$ we have an isomorphism $(F^m_{\tilde X_s})^*G\simeq G$, $G'=f_{s*}G$ is locally free and the relative evaluation $f_s^*f_{s*}G\to G$ is an isomorphism. In particular, we also have  $(F^m_{X_s})^*G'\simeq G'$.
		
		Let $\bar{s}$ be a geometric point lying over $s$. Then by \cite[Proposition 4.1.1]{Katz1972} $G_{\bar s}'$ gives rise to a continuous representation $\pi_1^{\et} (X_{\bar s})\to \GL _r (\FF_{p^m})$. 
        But by \cite[Lemma
        \href{https://stacks.math.columbia.edu/tag/0C0P}{0C0P}]{StacksProject},
        the specialization map $\pi_1^{\et} (X)\to \pi_1^{\et} (X_{\bar{s}})$ is surjective  and hence $ \pi_1^{\et} (X_{\bar{s}})=0$.
		This implies that $G'_{\bar s}$ is trivial. But then $G_{\bar{s}}$ is also trivial, which contradicts our assumption that
		$G$ is geometrically Gieseker stable.
	\end{proof}

\section*{Acknowledgements}

The first author was partially supported by the Polish National Centre  (NCN) grant {2021/41/B/ST1/03741}. The second author was supported
by Guangdong Basic and Applied Basic Research Foundation grant 
2025A1515012175. The authors would like to thank Bhargav Bhatt for kindly
permitting them to include his proof of Theorem \ref{h-descent-for-F-divided}
into this paper. Any errors introduced during the rephrasing of his
proof are solely the responsibility of the authors.

	\bibliographystyle{amsalpha}
	\bibliography{References}

\providecommand{\bysame}{\leavevmode\hbox to3em{\hrulefill}\thinspace}
\providecommand{\MR}{\relax\ifhmode\unskip\space\fi MR }
\providecommand{\MRhref}[2]{%
  \href{http://www.ams.org/mathscinet-getitem?mr=#1}{#2}
}
\providecommand{\href}[2]{#2}
\begin{thebibliography}{AMBL05}

\bibitem[AMBL05]{Alvarez-Montaner-Blickle-Lyubeznik2005}
Josep Alvarez-Montaner, Manuel Blickle, and Gennady Lyubeznik, \emph{Generators of {$D$}-modules in positive characteristic}, Math. Res. Lett. \textbf{12} (2005), no.~4, 459--473. \MR{2155224}

\bibitem[And74]{Andre1974}
Michel Andr\'{e}, \emph{Homologie des alg\`ebres commutatives}, Die Grundlehren der mathematischen Wissenschaften, Band 206, Springer-Verlag, Berlin-New York, 1974. \MR{352220}

\bibitem[Aut]{StacksProject}
The Stacks~Project Authors, \emph{Stacks project}.

\bibitem[Bha]{Bhatt}
B.~Bhatt, \emph{Notes on infinitesimal crystals in characteristic $p$}, preprint.

\bibitem[BK61]{Berger-Kunz1961}
Robert Berger and Ernst Kunz, \emph{\"{U}ber die {S}trucktur der {D}ifferentialmoduln von diskreten {B}ewertungsringen}, Math. Z. \textbf{77} (1961), 314--338. \MR{130891}

\bibitem[BS17]{Bhatt-Scholze2017}
Bhargav Bhatt and Peter Scholze, \emph{Projectivity of the {W}itt vector affine {G}rassmannian}, Invent. Math. \textbf{209} (2017), no.~2, 329--423. \MR{3674218}

\bibitem[BV15]{BV15}
Niels {Borne} and Angelo {Vistoli}, \emph{{The Nori fundamental gerbe of a fibered category.}}, {J. Algebr. Geom.} \textbf{24} (2015), no.~2, 311--353 (English).

\bibitem[Cha74]{Chase1974}
Stephen~U. Chase, \emph{On the homological dimension of algebras of differential operators}, Comm. Algebra \textbf{1} (1974), 351--363. \MR{498531}

\bibitem[dJ96]{deJong1996}
A.~J. de~Jong, \emph{Smoothness, semi-stability and alterations}, Inst. Hautes \'{E}tudes Sci. Publ. Math. (1996), no.~83, 51--93. \MR{1423020}

\bibitem[DM82]{Deligne-Milne1982}
Pierre Deligne and James~S. Milne, \emph{Tannakian categories}, Hodge Cycles, Motives, and Shimura Varieties, Lecture Notes in Mathematics, vol. 900, Springer-Verlag, Berlin-New York, 1982, pp.~101--228. \MR{654325}

\bibitem[dS07]{dos_Santos2007}
Jo\~{a}o Pedro~Pinto dos Santos, \emph{Fundamental group schemes for stratified sheaves}, J. Algebra \textbf{317} (2007), no.~2, 691--713. \MR{2362937}

\bibitem[EM10]{Esnault-Mehta2010}
H\'{e}l\`ene Esnault and Vikram Mehta, \emph{Simply connected projective manifolds in characteristic {$p>0$} have no nontrivial stratified bundles}, Invent. Math. \textbf{181} (2010), no.~3, 449--465. \MR{2660450}

\bibitem[ES16]{Esnault-Srinivas2016}
H\'{e}l\`ene Esnault and Vasudevan Srinivas, \emph{Simply connected varieties in characteristic {$p>0$}}, Compos. Math. \textbf{152} (2016), no.~2, 255--287, With an appendix by Jean-Beno\^{i}t Bost. \MR{3462553}

\bibitem[Esn13]{Esnault-Mehta-erratum}
H\'{e}l\`ene Esnault, \emph{{(Little)} {Erratum} to: Simply connected projective manifolds in characteristic $p>0$ have no nontrivial stratified bundles. \emph{Invent. Math.} {\bf 181} \rm{(2010), 449--465}}, 2013, preprint.

\bibitem[Fin25]{Fink2025}
Daniel Fink, \emph{Relative inverse limit perfection of derived commutative rings}, 2025, arXiv:2506.10626 [math.AG].

\bibitem[FL81]{Fulton-Lazarsfeld1981}
William Fulton and Robert Lazarsfeld, \emph{Connectivity and its applications in algebraic geometry}, Algebraic geometry ({C}hicago, {I}ll., 1980), Lecture Notes in Math., vol. 862, Springer, Berlin, 1981, pp.~26--92. \MR{644817}

\bibitem[Fog80]{Fogarty1980}
John Fogarty, \emph{K\"ahler differentials and {H}ilbert's fourteenth problem for finite groups}, Amer. J. Math. \textbf{102} (1980), no.~6, 1159--1175. \MR{595009}

\bibitem[Ful98]{Fulton1998}
William Fulton, \emph{Intersection theory}, second ed., Ergebnisse der Mathematik und ihrer Grenzgebiete. 3. Folge. A Series of Modern Surveys in Mathematics [Results in Mathematics and Related Areas. 3rd Series. A Series of Modern Surveys in Mathematics], vol.~2, Springer-Verlag, Berlin, 1998. \MR{1644323}

\bibitem[Gab05]{Gabber2005}
Ofer Gabber, \emph{Finiteness theorems for étale cohomology of excellent schemes}, 2005, Conférence en l'honneur de Pierre Deligne à l'occasion de son soixante-et-unième anniversaire, Institute for Advanced Study, Princeton. (Voir annexe B.), https://www.cmls.polytechnique.fr/perso/laszlo/gdtgabber/abelien.pdf.

\bibitem[Gie75]{Gieseker1975}
D.~Gieseker, \emph{Flat vector bundles and the fundamental group in non-zero characteristics}, Ann. Scuola Norm. Sup. Pisa Cl. Sci. (4) \textbf{2} (1975), no.~1, 1--31. \MR{382271}

\bibitem[Gro68]{Grothendieck1968}
A.~Grothendieck, \emph{Crystals and the de {R}ham cohomology of schemes}, Dix expos\'{e}s sur la cohomologie des sch\'{e}mas, Adv. Stud. Pure Math., vol.~3, North-Holland, Amsterdam, 1968, Notes by I. Coates and O. Jussila, pp.~306--358. \MR{269663}

\bibitem[Gro70]{Grothendieck1970}
Alexander Grothendieck, \emph{Repr\'{e}sentations lin\'{e}aires et compactification profinie des groupes discrets}, Manuscripta Math. \textbf{2} (1970), 375--396. \MR{262386}

\bibitem[Gro05]{SGA2}
\bysame, \emph{Cohomologie locale des faisceaux coh\'erents et th\'eor\`emes de {L}efschetz locaux et globaux ({SGA} 2)}, Documents Math\'ematiques (Paris) [Mathematical Documents (Paris)], vol.~4, Soci\'et\'e{} Math\'ematique de France, Paris, 2005, S\'eminaire de G\'eom\'etrie Alg\'ebrique du Bois Marie, 1962., Augment\'e{} d'un expos\'e{} de Mich\`ele Raynaud. [With an expos\'e{} by Mich\`ele Raynaud], Revised reprint of the 1968 French original. \MR{2171939}

\bibitem[HLP23]{Halpern-Leistner-Preygel2023}
Daniel Halpern-Leistner and Anatoly Preygel, \emph{Mapping stacks and categorical notions of properness}, Compos. Math. \textbf{159} (2023), no.~3, 530--589. \MR{4560539}

\bibitem[Hru04]{Hrushovski2004}
Ehud Hrushovski, \emph{The elementary theory of {F}robenius automorphisms}, 2004, arXiv:math/0406514v2.

\bibitem[ILO14]{Illusie-Laszlo-Orgogozo2014}
Luc Illusie, Yves Laszlo, and Fabrice Orgogozo (eds.), \emph{Travaux de {G}abber sur l'uniformisation locale et la cohomologie \'{e}tale des sch\'{e}mas quasi-excellents}, Soci\'{e}t\'{e} Math\'{e}matique de France, Paris, 2014, S\'{e}minaire \`a l'\'{E}cole Polytechnique 2006--2008. [Seminar of the Polytechnic School 2006--2008], With the collaboration of Fr\'{e}d\'{e}ric D\'{e}glise, Alban Moreau, Vincent Pilloni, Michel Raynaud, Jo\"{e}l Riou, Beno\^{i}t Stroh, Michael Temkin and Weizhe Zheng, Ast\'{e}risque No. 363-364 (2014). \MR{3309086}

\bibitem[Kat70]{Katz1970}
Nicholas~M. Katz, \emph{Nilpotent connections and the monodromy theorem: {A}pplications of a result of {T}urrittin}, Inst. Hautes \'{E}tudes Sci. Publ. Math. (1970), no.~39, 175--232. \MR{291177}

\bibitem[Kat73]{Katz1972}
\bysame, \emph{{$p$}-adic properties of modular schemes and modular forms}, Modular functions of one variable, {III} ({P}roc. {I}nternat. {S}ummer {S}chool, {U}niv. {A}ntwerp, {A}ntwerp, 1972), Lecture Notes in Math., Vol. 350, Springer, Berlin-New York, 1973, pp.~69--190. \MR{447119}

\bibitem[Kin15]{Kindler2015}
Lars Kindler, \emph{Regular singular stratified bundles and tame ramification}, Trans. Amer. Math. Soc. \textbf{367} (2015), no.~9, 6461--6485. \MR{3356944}

\bibitem[KN80]{Kimura-Niitsuma1980}
Tetsuzo Kimura and Hiroshi Niitsuma, \emph{Regular local ring of characteristic {$p$} and {$p$}-basis}, J. Math. Soc. Japan \textbf{32} (1980), no.~2, 363--371. \MR{567425}

\bibitem[KN82]{Kimura-Niitsuma1982}
\bysame, \emph{On {K}unz's conjecture}, J. Math. Soc. Japan \textbf{34} (1982), no.~2, 371--378. \MR{651278}

\bibitem[Kun86]{Kunz-book}
Ernst Kunz, \emph{K\"{a}hler differentials}, Advanced Lectures in Mathematics, Friedr. Vieweg \& Sohn, Braunschweig, 1986. \MR{864975}

\bibitem[Lan04]{Langer2004}
Adrian Langer, \emph{Semistable sheaves in positive characteristic}, Ann. of Math. (2) \textbf{159} (2004), no.~1, 251--276. \MR{2051393}

\bibitem[MP]{Ma-Polstra}
Linquan Ma and Thomas Polstra, \emph{{$F$}-singularities: a commutative algebra approach}, book available at {\tt https://www.math.purdue.edu/$\sim$ma326/F-singularitiesBook.pdf}.

\bibitem[SGA71]{SGA6}
\emph{Th\'eorie des intersections et th\'eor\`eme de {R}iemann-{R}och}, Lecture Notes in Mathematics, vol. Vol. 225, Springer-Verlag, Berlin-New York, 1971, S\'eminaire de G\'eom\'etrie Alg\'ebrique du Bois-Marie 1966--1967 (SGA 6), Dirig\'e{} par P. Berthelot, A. Grothendieck et L. Illusie. Avec la collaboration de D. Ferrand, J. P. Jouanolou, O. Jussila, S. Kleiman, M. Raynaud et J. P. Serre. \MR{354655}

\bibitem[SGA03]{SGA1}
\emph{Rev\^etements \'etales et groupe fondamental ({SGA} 1)}, Documents Math\'ematiques (Paris) [Mathematical Documents (Paris)], vol.~3, Soci\'et\'e{} Math\'ematique de France, Paris, 2003, S\'eminaire de g\'eom\'etrie alg\'ebrique du Bois Marie 1960--61. [Algebraic Geometry Seminar of Bois Marie 1960-61], Directed by A. Grothendieck, With two papers by M. Raynaud, Updated and annotated reprint of the 1971 original [Lecture Notes in Math., 224, Springer, Berlin; MR0354651 (50 \#7129)]. \MR{2017446}

\bibitem[SR72]{Saavedra72}
Neantro Saavedra~Rivano, \emph{Cat\'egories {T}annakiennes}, Lecture Notes in Mathematics, vol. Vol. 265, Springer-Verlag, Berlin-New York, 1972. \MR{338002}

\bibitem[Tyc88]{Tyc1988}
Andrzej Tyc, \emph{Differential basis, {$p$}-basis, and smoothness in characteristic {$p>0$}}, Proc. Amer. Math. Soc. \textbf{103} (1988), no.~2, 389--394. \MR{943051}

\bibitem[TZ19]{TZ1}
Fabio Tonini and Lei Zhang, \emph{Algebraic and {N}ori fundamental gerbes}, J. Inst. Math. Jussieu \textbf{18} (2019), no.~4, 855--897. \MR{3963521}

\end{thebibliography}

\end{document}